\documentclass[11pt,a4paper,reqno]{amsart}
\usepackage{a4wide}
\usepackage{amsaddr}
\usepackage{amsmath}
\usepackage{amsfonts, amssymb, amsthm}
\usepackage{paralist}
\usepackage[colorlinks=true]{hyperref}
\hypersetup{urlcolor=blue, citecolor=red}

\makeatletter
\renewcommand{\email}[2][]{%
	\ifx\emails\@empty\relax\else{\g@addto@macro\emails{,\space}}\fi%
	\@ifnotempty{#1}{\g@addto@macro\emails{\textrm{(#1)}\space}}%
	\g@addto@macro\emails{#2}%
}
\makeatother

\newtheorem{theorem}{Theorem}[section]
\newtheorem{corollary}[theorem]{Corollary}
\newtheorem{lemma}[theorem]{Lemma}

\theoremstyle{definition}

\newtheorem{remark}[theorem]{Remark}

\numberwithin{equation}{section}

\newcommand{\R}{{\mathbb R}}

\begin{document}
\title[Nonlinear Schr\"odinger equation with critical combined powers nonlinearity]
{Asymptotic profiles for a nonlinear Schr\"odinger equation with critical combined powers nonlinearity}
\author{Shiwang Ma}\email{shiwangm@nankai.edu.cn}
\address{School of Mathematical Sciences and LPMC, Nankai University\\ 
	Tianjin 300071, China}
%https://orcid.org/0000-0003-2755-2967
\author{Vitaly Moroz}\email{v.moroz@swansea.ac.uk}
\address{Department of Mathematics, Swansea University\\ 
	Fabian Way, Swansea SA1 8EN, 	Wales, UK}
%https://orcid.org/0000-0003-3302-8782



\date{\today}

\begin{abstract}
We study asymptotic behaviour of positive ground state solutions of the nonlinear Schr\"odinger equation
\[
-\Delta u+u=u^{2^*-1}+\lambda u^{q-1} 
\quad {\rm in} \ \mathbb R^N,
 \eqno{(P_\lambda)}
\]
where $N\ge 3$ is an integer, $2^*=\frac{2N}{N-2}$ is the Sobolev critical exponent,  $2<q<2^*$ and 
$\lambda>0$ is a parameter. It is known that as $\lambda\to 0$, after {\em a rescaling} the ground state solutions of $(P_\lambda)$ converge to a particular solution of the critical Emden-Fowler equation $-\Delta u=u^{2^*-1}$. We establish a novel sharp asymptotic characterisation of such a rescaling, which depends in a non-trivial way on the space dimension $N=3$, $N=4$ or $N\ge 5$. We also discuss a connection of these results with a mass constrained problem associated to $(P_\lambda)$.
Unlike previous work of this type, our method is based on the Nehari-Poho\v zaev manifold minimization, which allows to control the $L^2$--norm of the groundstates. 
\end{abstract}

\maketitle
%\tableofcontents

\section{Introduction and notations}

We study standing--wave solutions of the nonlinear Schr\"odinger equation with attractive double--power nonlinearity
\begin{equation}\label{eNLS}
i\psi_t=\Delta \psi+|\psi|^{q-2}\psi+|\psi|^{p-2}\psi\quad\text{in $\R^N\times\R$}
\end{equation}
where $N\ge 3$ is an integer and $2<q<p$.
A theory of NLS with combined power nonlinearities was developed by Tao, Visan and Zhang \cite{Tao} and attracted a lot of attention during the past decade (cf. \cite{Akahori-3,Akahori-2,Coles} and further references therein). 

A standing--wave solutions of \eqref{eNLS} with a frequency $\omega>0$ is a finite energy solution in the form
$$\psi(t,x)=e^{-i\omega t}Q(x).$$
After a rescaling
$$Q(x)=\omega^\frac{1}{p-2}u(\sqrt{\omega}x),$$
we obtain the equation for $u$ in the form
\begin{equation}\label{eNLS+}
-\Delta u+ u=|u|^{p-2}u+\lambda |u|^{q-2}u\quad\text{in $\R^N$},
\end{equation}
where $\lambda=\omega^{-\frac{p-q}{p-2}}>0$. 

When $p\le 2^*$, where $2^*=\frac{2N}{N-2}$ is the Sobolev critical exponent,
weak solutions of \eqref{eNLS+} correspond to critical points of the associated energy functional $I_\lambda:  H^1(\mathbb  R^N)\to \mathbb R$, defined by
$$
I_\lambda(u):=\frac{1}{2}\int_{\mathbb R^N}\left(|\nabla u|^2+|u|^2\right)-\frac{1}{p}\int_{\mathbb R^N}|u|^p-\frac{\lambda}{q}\int_{\mathbb R^N}|u|^q.
$$
By a ground state solution of \eqref{eNLS+} we understand a solution  $u_\lambda\in H^1(\mathbb  R^N)$ such that $I_\lambda(u_\lambda)\le I_\lambda(u)$ for every nontrivial solution $u$ of \eqref{eNLS+}.

In the subcritical case $p<2^*$, the existence of a positive radially symmetric exponentially decaying ground state solution of \eqref{eNLS+} is the result of Berestycki and Lions \cite{Berestycki-1}. If $2^*\le q<p$ there are no finite energy solutions of \eqref{eNLS+}, which follows from Poh\v zaev identity.

In this paper we are interested in the critical case $p=2^*$.
We study the problem 
\[
-\Delta u+u=u^{2^*-1}+\lambda u^{q-1},\qquad\text{$u>0$ in $\R^N$,}\eqno(P_\lambda)
\]
where $q\in (2,2^*)$ and $\lambda>0$ is a parameter. 
The following result gives a characterisation of the existence of ground states for $(P_\lambda)$. 

\begin{theorem}\label{t0}
Problem $(P_\lambda)$ admits a positive radially symmetric exponentially decreasing ground state solution $u_\lambda\in H^1(\mathbb R^N)\cap C^2(\R^N)$ provided that: 
\begin{itemize}
\item $N\ge 4$, $q\in(2,2^*)$ and $\lambda>0$;
\item $N=3$, $q\in (4,6)$ and $\lambda >0$;
\item $N=3$ and $q\in (2,4]$ and $\lambda$ is sufficiently large.
\end{itemize}
\end{theorem}  

For $N\ge 4$, Theorem \ref{t0} is established by Akahori, Ibrahim, Kikuchi and Nawa \cite{Akahori}, Alves, Souto and Montenegro \cite{Alves-1} and Liu, Liao and Tang \cite{Liu-1}. In the case $N=3$, Theorem \ref{t0} is proved in the above mentioned papers for $q\in (2,6)$ and large $\lambda>0$. 
Theorem \ref{t0} for $N=3$, $q\in (4,6)$ and every $\lambda>0$ was proved in Zhang and Zou \cite[Theorem 1.1]{Zhang-1} 
(see  also Li and Ma \cite{Li-1} or Akahori et al. \cite[Proposition 1.1]{Akahori-2}).

Very recently, Akahori, Ibrahim, Kikuchi and Nawa \cite{Akahori-22}, and Wei and Wu \cite{Wei-2} refined the results concerning the existence and non-existence of ground states to $(P_\lambda)$ when $N=3$. Although their definition of the ground state is different from that in our paper, they established the existence of a $\lambda_*>0$ such that $(P_\lambda)$ has a ground state if $\lambda>\lambda_*$ and no ground state if $\lambda<\lambda_*$ when $N = 3$ and $q \in (2, 4]$. Moreover, when $N=3$ and $\lambda=\lambda_*$, $(P_\lambda)$ has a ground state if $q\in(2,4)$. 

Concerning the uniqueness, Akahori et al. \cite{Akahori-2,Akahori-2+,Akahori-3} and Coles and Gustafson \cite{Coles} proved that the radial ground state $u_\lambda$ is unique and nondegenerate for all small $\lambda>0$ when $N\ge 5$ and $q\in(2,2^*)$ \cite[Theorem 1.1]{Akahori-2} or $N=3$ and $q\in(4,2^*)$ \cite{Coles}, \cite[Theorem 1.1]{Akahori-2+}; and for all large $\lambda$ when $N\ge 3$ and $2+4/N<q<2^*$ \cite[Proposition 2.4]{Akahori-3}.
Very recently, Akahori and Murata  \cite{Akahori-22-NoDEA,Akahori-22-arxiv} established the uniqueness and nondegeneracy of the ground state solutions for small $\lambda>0$ in the case $N = 4$.

In general, the uniqueness of positive radial solutions of $(P_\lambda)$ is not expected. D\'avila, del Pino and Guerra \cite{Davila} constructed multiple positive solutions of \eqref{eNLS+} for a sufficiently large $\lambda$ and slightly subcritical $p<2^*$. A numerical simulation in the same paper suggested nonuniqueness in the critical case $p=2^*$. Wei and Wu \cite{Wei-2} recently proved that there exist two positive solutions to $(P_\lambda)$ when $N = 3$, $q \in (2, 4)$ and $\lambda > 0$ is sufficiently large, as \cite{Davila} has  suggested. Chen, D\'avila and Guerra \cite{Davila-2} proved the existence of arbitrary large number of bubble tower positive solutions of  \eqref{eNLS+} in the slightly supercritical case when $q<2^*<p=2^*+\varepsilon$, provided that $\varepsilon>0$ is sufficiently small.
However, if $3\le N\le 6$ and $\frac{N+2}{N-2}<q<2^*$ then Pucci and Serrin \cite[Theorem 1]{Pucci} proved that $(P_\lambda)$ has at most one positive radial solution (see also \cite[Theorem C.1]{Akahori}).

Existence of a positive radial solution to \eqref{eNLS+} in the supercritical case $2<q<2^*\le p$ for sufficiently large $\lambda$ was established earlier by Ferrero and Gazzola \cite[Theorem 5]{Ferrero} using ODE's methods, however the variational characterisation of these solutions seems open. They also proved that for $2<q<2^*<p$ and small $\lambda>0$ equation \eqref{eNLS+} has no positive solutions.

\smallskip

Before we formulate the result in this paper we shall clarify the notations.
\smallskip

\paragraph{\bf Notations.} 
Throughout the paper, we assume $N\geq 3$. 
The standard norm on the Lebesgue space $L^p(\R^N)$ is denoted by $\|\cdot\|_p$. 
The space $H^1(\R^{N})$ is the usual Sobolev space with the norm $\|u\|_{H^1(\mathbb{R}^N)}=\|\nabla u\|_2+\|u\|_2$, while
$H_r^1(\R^N)=\{u\in H^1(\R^N): u\;\text{is radially  symmetric}\}$.
The homogeneous Sobolev space $D^1(\mathbb{R}^{N})$ is defined as the completion of $C^\infty_c(\R^N)$ with respect to the norm $\|\nabla u\|_2$.  

For any  small $\lambda>0$, any $q\in (2,2^*)$, and two nonnegative functions $f(\lambda, q)$ and  $g(\lambda, q)$, 
throughout the paper we write:

\begin{itemize}
	\item
	$f(\lambda,q)\lesssim g(\lambda,q)$ or $g(\lambda,q)\gtrsim f(\lambda,q)$ if there exists a positive constant $C$ independent of $\lambda$ and $q$ such that $f(\lambda,q)\le Cg(\lambda,q)$,
	\item
	$f(\lambda,q)\sim g(\lambda,q)$ if $f(\lambda,q)\lesssim g(\lambda,q)$ and $f(\lambda,q)\gtrsim g(\lambda,q)$.
\end{itemize}
$B_R$ denotes the open ball in $\R^N$ with radius $R>0$ and centred at the origin,  $|B_R|$ and $B_R^c$ denote its Lebesgue measure and its complement in $\R^N$, respectively.   
As usual, $c,c_1$ etc., denote positive constants which are independent of $\lambda$ and whose exact values are irrelevant.

\section{Main result}
In this paper we are interested in the limit asymptotic profile of the ground states $u_\lambda$ of the critical problem $(P_\lambda)$, and in the asymptotic behaviour of different norms of $u_\lambda$, as $\lambda\to 0$ and $\lambda\to\infty$.
Of particular importance is the $L^2$--mass of the ground state
$$M(\lambda):=\|u_\lambda\|_2^2,$$
which plays a key role in the analysis of stability of the corresponding standing--wave solution of the time--dependent NLS \eqref{eNLS}, and in the study of the mass constrained problems associated to $(P_\lambda)$,
cf.~Lewin and Nodari \cite[Section 3.2]{Lewin-1} and Section \ref{s-mass} below for a discussion.

In the subcritical case $p<2^*$, it is intuitively clear and not difficult to show (using e.g. Lyapunov--Schmidt type arguments) that as $\lambda\to 0$, ground states of \eqref{eNLS+} converge to the unique radial positive ground state of the limit equation
\begin{equation}\label{eNLS+0}
-\Delta u+ u=|u|^{p-2}u\quad\text{in $\R^N$}.
\end{equation}
In the critical case $p=2^*$, by Poho\v zaev identity, the {\em formal} limit equation \eqref{eNLS+0} has no nontrvial finite energy solutions. In fact, we will see later that $u_\lambda$ converges as $\lambda\to 0$ to a multiple of the delta-function at the origin.

Recently Akahori et al.~\cite[Proposition 2.1]{Akahori-2} proved that {\em after a rescaling}, the correct limit equation for $(P_\lambda)$ as $\lambda\to 0$ is given by the critical Emden-Fowler equation
\begin{equation}\label{e12}
-\Delta U=U^{2^*-1} \quad\text{in $\R^N$}.
\end{equation}
Recall that all radial solutions of \eqref{e12} are given by the Talenti function 
\begin{equation}\label{e13}
U_1(x):=[N(N-2)]^{\frac{N-2}{4}}\left(\frac{1}{1+|x|^2}\right)^{\frac{N-2}{2}}
\end{equation}
and the family of its rescalings 
\begin{equation}\label{e14}
U_\rho(x):=\rho^{-\frac{N-2}{2}}U_1(x/\rho),  \quad \rho>0.
\end{equation}
Note that while $(P_\lambda)$ and the associated energy $I_\lambda$ are well--posed in $H^1(\R^N)$, the limit critical Emden-Fowler equation \eqref{e12} is well--posed in $D^1(\R^N)\not\subset H^1(\R^N)$. Moreover, in the dimensions $N=3,4$ the ground states $U_\rho\not\in H^1(\R^N)$, so small perturbation arguments are not (easily) available for the study of limit behaviour of $u_\lambda$.

Akahori et al.~\cite[Proposition 2.1]{Akahori-2} proved, using variational methods, that the rescaled family of ground state solutions of $(P_\lambda)$, defined as
\begin{equation}\label{resc-imp}
\tilde{u}_\lambda(x):=\mu_\lambda^{-1}u_\lambda\big(\mu_\lambda^{-\frac2{N-2}}x\big),\qquad\mu_\lambda:=u_\lambda(0)=\|u_\lambda\|_\infty
\end{equation}
converges as $\lambda\to 0$ in $D^1(\R^N)$ to the $U_{\rho_*}$, where $\|U_{\rho_*}\|_\infty=1$.
This result was used in the proof of the uniqueness and nondegenaracy of the ground states of $(P_\lambda)$ for $N\ge 5$ in \cite{Akahori-2}, and for $N=3$ in \cite{Akahori-2+}. 
%It is pointed out in \cite{Akahori-2+} that the uniqueness and nondegenaracy in the case $N=4$ is currently open. 
Very recently, Akahori and Murata  \cite{Akahori-22-NoDEA,Akahori-22-arxiv} obtained the uniqueness and nondegeneracy of the ground state solutions in the case $N = 4$.
The rescaling $\mu_\lambda$ in \eqref{resc-imp} is implicit.

Our main result in this work is an explicit asymptotic characterisation of a rescaling which ensures the convergence of ground states of $(P_\lambda)$ to a ground state of the critical Emden--Fowler equation \eqref{e12}. More precisely, we prove the following.

\begin{theorem}\label{t1}
%Let $N\ge 4$ and $2<q<2^*$ or	$N=3$ and $4<q<6$.
Let $\{u_\lambda\}$ be a family of ground states of $(P_\lambda)$. 
\begin{itemize}
\item[$(a)$]
If $N\ge 5$ and $q\in(2,2^*)$, then for small $\lambda>0$
\begin{equation}
u_\lambda(0)\sim \lambda^{-\frac{1}{q-2}},
\end{equation}
\begin{equation}
\|\nabla u_\lambda\|^2_2\sim \|u_\lambda\|_{2^*}^{2^*}\sim 1, \quad \|u_\lambda\|_2^2\sim (2^*-q)\lambda^{\frac{2^*-2}{q-2}},  \quad 
\|u_\lambda\|^q_q\sim \lambda^{\frac{2^*-q}{q-2}}.
\end{equation}
Moreover, as $\lambda\to 0$, the rescaled family of ground states 
\begin{equation}\label{e15}
v_\lambda(x)=\lambda^\frac{1}{q-2}u_\lambda\big(\lambda^\frac{2^*-2}{2(q-2)}x\big), 
\end{equation}
%satisfies 
%\begin{equation}
%\|\nabla v_\lambda\|^2_2\sim \|v_\lambda\|_{2^*}^{2^*} \sim\|v_\lambda\|^q_q\sim 1,\  \|v_\lambda\|_2^2\sim 2^*-q,
%\end{equation}
%and as $\lambda\to 0$, $v_\lambda$ 
converges  to $U_{\rho_0}$ in $H^1(\mathbb R^N)$ with 
\begin{equation}
\rho_0=\left(\frac{2(2^*-q)\int_{\mathbb R^N}|U_1|^q}{q(2^*-2)\int_{\mathbb R^N}|U_1|^2}\right)^{\frac{2^*-2}{2(q-2)}},
\end{equation}
and the convergence rate is described by the relation
\begin{equation}\label{e-conv5}
\|\nabla U_{\rho_0}\|_2^2-\|\nabla v_\lambda\|_2^2\sim (q-2)\lambda^{\frac{2^*-2}{q-2}}.
\end{equation}
\smallskip

\item[$(b)$]
If $N=4$ and $q\in (2,4)$ or $N=3$ and $q\in (4,6)$, then for small $\lambda>0$
\begin{equation}\label{2-11}
u_\lambda(0)\sim\left\{
\begin{array}{lcl}
\lambda^{-\frac{N-2}{2(q-2)}}(\ln\frac{1}{\lambda})^{\frac{N-2}{2(q-2)}}&\text{if}&N=4,\smallskip\\
\lambda^{-\frac{N-2}{q-4}}&\text{if}&N=3,
\end{array}\right.
\end{equation}
\begin{equation}
\|\nabla u_\lambda\|^2_2\sim \|u_\lambda\|_{2^*}^{2^*} \sim 1,
\end{equation}
\begin{equation}
\|u_\lambda\|_2^2\sim\left\{
\begin{array}{lcl}
\lambda^{\frac{2}{q-2}}(\ln\frac{1}{\lambda})^{-\frac{4-q}{q-2}}&\text{if}&N=4,\smallskip\\
\lambda^{\frac{2}{q-4}}&\text{if}&N=3,
\end{array}\right.
\end{equation}
\begin{equation}
\|u_\lambda\|_q^q\sim\left\{
\begin{array}{lcl}
\lambda^{\frac{4-q}{q-2}}(\ln\frac{1}{\lambda})^{-\frac{4-q}{q-2}}&\text{if}&N=4,\smallskip\\
\lambda^{\frac{6-q}{q-4}}&\text{if}&N=3.
\end{array}\right.
\end{equation}
Moreover,  there exists $\xi_\lambda\in (0,+\infty)$ verifying  
\begin{equation}\label{e215}
\xi_\lambda\sim\left\{
\begin{array}{lcl}
\lambda^{\frac{1}{q-2}}(\ln\frac{1}{\lambda})^{-\frac{1}{q-2}}&\text{if}&N=4,\smallskip\\
\lambda^{\frac{2}{q-4}}&\text{if}&N=3,
\end{array}\right.
\end{equation}
such that as  $\lambda\to 0$, the rescaled family of ground states
\begin{equation}\label{e16}
w_\lambda(x)=\xi_\lambda^{\frac{N-2}{2}}u_\lambda(\xi_\lambda x),
\end{equation}
%satisfies 
%\begin{equation}
%\|\nabla w_\lambda\|^2_2\sim \|w_\lambda\|_{2^*}^{2^*}\sim \|w_\lambda\|_q^q\sim 1, \quad 
%\  \|w_\lambda\|_2^2\sim \left\{
%\begin{array}{lcl}
%\ln\frac{1}{\lambda}&\text{if}&N=4,\smallskip\\
%\lambda^{-\frac{2}{q-4}}&\text{if}&N=3,
%\end{array}\right.
%\end{equation}
%and as $\lambda\to 0$, $w_\lambda$ 
converges  to $U_1$ in $D^{1}(\mathbb R^N)\cap L^q(\mathbb R^N)$ , and the convergence rate is described by the relation
\begin{equation}\label{e-conv34}
\|\nabla U_1\|_2^2-\|\nabla w_\lambda\|_2^2\sim \left\{
\begin{array}{lcl}
 \lambda^{\frac{2}{q-2}}\left(\ln \frac{1}{\lambda}\right)^{-\frac{4-q}{q-2}} &\text{if}&N=4,\smallskip\\
\lambda^{\frac{2}{q-4}}&\text{if}&N=3.
 \end{array}
 \right.
\end{equation}
\end{itemize}
\end{theorem}

Similar type of results were recently obtained by Wei and Wu \cite{Wei-1,Wei-2}. In \cite{Wei-2} the authors study solutions of $(P_\lambda)$ in the case $N=3$ and $q\in(2,4)$. In particular, \cite[Theorem 1.2 and Propostion 2.4]{Wei-2} proves that for sufficiently large $\mu$ there exist a ground state and a {\em blow-up} positive radial solution of $(P_\lambda)$, and derives asymptotic estimates of type \eqref{2-11} on these two solutions. These results complement Theorem \ref{t1} above. 
In \cite{Wei-1} the authors study normalised solutions of $(P_\lambda)$ for $N\ge 3$ and general range $q\in(2,2^*)$. In \cite[Theorem 1.2 and Propostion 2.4]{Wei-1} they show convergence up to a rescaling of the mountain--pass type normalised solution of $(P_\lambda)$ with a fixed mass to a normalised solution of the Emden--Fowler equation \eqref{e12} and derive asymptotic estimates of the rescaling similar to the results in Theorem \ref{t1}. It is not known in general (cf. Section 2) whether or not normalised solutions in \cite{Wei-1} are (rescalings of) ground states in Theorem \ref{t1}. In fact, comparison of estimates in \cite{Wei-1} and Theorem \ref{t1} could potentially help to study this question. The techniques in our work and in \cite{Wei-1,Wei-2} are different.

Asymptotic characterisation of ground states of the equation with a double--well nonlinearity in the form
\begin{equation}\label{eNLS+-}
-\Delta u+ \omega u=|u|^{p-2}u-|u|^{q-2}u\quad\text{in $\R^N$},
\end{equation}
with $\omega>0$ and $2<q<p<+\infty$ was obtained by Moroz and Muratov \cite{Moroz-1}, and by Lewin and Nodari \cite{Lewin-1}.
Our proof of Theorem \ref{t1} is inspired by \cite{Moroz-1} yet the techniques in the present work are different.
While the arguments in \cite{Moroz-1} are based on the Berestycki--Lions variational approach \cite{Berestycki-1}, the proofs in this work use minimization over Nehari manifold combined with Pohozaev's identity estimates, and the Concentration Compactness Principle. The advantage of the Nehari--Poho\v zaev approach is that it allows to include the control the $L^2$--norm of the ground states, which is essential in the study of the mass constrained problems associated to $(P_\lambda)$. Our method could be extended to nonlinear Hartree type equations with nonlocal convolution terms which include competing scaling symmetries \cite{MM-prep} and nonlocal Kirchhoff equations \cite{MM-Kirch}, while the Berestycki--Lions approach seems to be limited to local equations only.
\smallskip

In the case $\lambda\to\infty$, the explicit rescaling
\begin{equation}\label{eq-rinf}
v(x)=\lambda^{\frac{1}{q-2}}u(x)
\end{equation}
becomes relevant. Clearly, \eqref{eq-rinf} transforms $(P_\lambda)$ into the equivalent equation
\[
-\Delta v+v=\lambda^{-\frac{2^*-2}{q-2}}v^{2^*-1}+v^{q-1} \quad \text{in $\R^N$}.
\eqno{(R_\lambda)}
\]
This suggests that as $\lambda\to\infty$ the limit equation for $(R_\lambda)$ is given by the equation
\begin{equation}\label{eq-rinf-0}
-\Delta v+v=v^{q-1}\quad \text{in $\R^N$},
\end{equation}
which has the unique positive radial solution $v_\infty\in H^1(\mathbb R^N)\cap C^2(\R^N)$. 
For completeness, we formulate the following result, which was proved by Fukuizumi \cite[Lemma 4.2]{Fukuizumi} (see also \cite[Proposition 2.3]{Akahori-3}).

\begin{theorem}\label{t3}
	Let $N\ge 3$, $q\in(2,2^*)$ and $\{u_\lambda\}$ be a family of ground states of $(P_\lambda)$. Then 
	as $\lambda\to +\infty$, the rescaled family of ground states 
	\begin{equation}\label{e17}
	v_\lambda(x)=\lambda^\frac{1}{q-2}u_\lambda(x)
	\end{equation}
	converges in $H^1(\mathbb R^N)$ to $v_\infty$. 
	Moreover, the convergence rate is described by the relation
 \begin{equation}
	\|v_\infty\|^2_{H^1(\mathbb R^N)}-\|v_\lambda\|^2_{H^1(\mathbb R^N)}=\frac{1}{q-2} \lambda^{-\frac{2^*-2}{q-2}}(1+o(1)).
	\end{equation}
\end{theorem}

The Nehari--Poho\v zaev variational arguments developed in this work can be adapted to show that the statement of Theorem \ref{t3} remains valid also for the equation \eqref{eNLS+} in whole range case of admissible exponents $2<q<p\le 2^*$. We omit the details, as these mostly repeat (in simplified form) the arguments in our proof of Theorem \ref{t1} in the case $N\ge 5$.

In the rest of the paper we concentrate on the case $\lambda\to 0$. In Section \ref{s3} we obtain several preliminary estimates. In Section \ref{s4} we prove Theorem \ref{t1}. However, before we proceed with the proof of Theorem \ref{t1}, in the next section section we discuss a connection with the mass constrained problem.

\section{A connection with the mass constrained problem}\label{s-mass}

Consider the energy
\begin{eqnarray*}
	J(v):=\frac{1}{2}\int |\nabla v|^2dx-\frac{1}{q}\int|v|^q dx
	-\frac{1}{p}\int|v|^p dx,
\end{eqnarray*}
constrained on 
$$S_\rho:=\big\{v\in H^1(\R^N)\,:\,\|v\|_{L^2}=\rho\big\}.$$
For $2<q<p\le 2^*$, critical points of $J$ on $S_\rho$ satisfy
\begin{equation}\label{eNLS++}
	-\Delta v+ \omega_\rho v=|v|^{p-2}u+|v|^{q-2}v\quad\text{in $\R^N$},
\end{equation}
where $\omega_\rho\in\R$ is an unknown Lagrange multiplier.
A ground state of $J$ on $S_\rho$ is a minimal energy critical point of $J$ on $S_\rho$. 

According to \cite[Theorem 1.1]{Soave} (see also \cite[Theorem 1.4]{Li}), for all $N\ge 3$, $2<q<2^*$, and for all sufficiently small $\rho>0$, the energy $J$ admits a ground state $v_\rho$ on $S_\rho$. The ground state $v_\rho$ is positive, radially symmetric and satisfies \eqref{eNLS++} with an $\omega_\rho>0$. When $2<q<2+4/N$ the ground state $v_\rho$ is a local minimum of $J$ on $S_\rho$, while for $2+4/N\le q<2^*$ the ground state $v_\rho$ is a mountain--pass type critical point of $J$ on $S_\rho$.

Recall that \eqref{eNLS++} is equivalent to $(P_\lambda)$ after a rescaling
%\begin{equation}\label{e-olambda}
%	\omega=\lambda^{-\frac{4}{(N-2)(2^*-q)}},\qquad v(x)=\omega^\frac{N-2}{4}u(\sqrt{\omega}x)
%\end{equation}
\begin{equation}\label{e-olambda}
	\lambda_\rho:=\omega_\rho^{-\frac{(N-2)(2^*-q)}{4}},\qquad v(x)=\omega_\rho^\frac{N-2}{4}u(\sqrt{\omega_\rho}x)
\end{equation}
and thus the results of Theorem \ref{t1} in principle could be applicable to \eqref{eNLS++}.
Caution however is needed as it is a-priori unknown (and generally speaking isn't always true \cite{Killip,Lewin-1}) if a ground state of $J$ on $S_\rho$ corresponds, after the rescaling \eqref{e-olambda}, to a ground state of the unconstrained problem $(P_{\lambda_\rho})$.
Recall however that when $3\le N\le 6$ and $q\in(2^*-1,2^*)$, equation $(P_\lambda)$ has at most one positive radial solution \cite[Theorem 1]{Pucci} (see also \cite[Theorem C.1]{Akahori}).
Hence a positive ground state of $J$ on $S_\rho$, when it exists, must coincide after the rescaling \eqref{e-olambda} with the unique positive solution of $(P_{\lambda_\rho})$. Even in this uniqueness scenario, the relation $\rho\to\omega_\rho$ (and hence $\rho\to\lambda_\rho$) is apriori unknown.
It turns out however that the asymptotic of $\lambda_\rho$ as $\rho\to 0$ can be recovered via the Poho\v zaev-Nehari identities and the estimates of the $L^q$-norm of $u_{\lambda_\rho}$ from Theorem \ref{t1}.
The following result links Theorem \ref{t1} with the mass constrained problem.

\begin{theorem}
	Assume that $3\le N\le 6$ and $q\in(2^*-1,2^*)$. Let $\rho\to 0$, and $v_\rho\in S_\rho$ be the the ground state of $J$ on $S_\rho$. Then 
	$$v_\rho(x)=\lambda_\rho^{-\frac{1}{2^*-q}}u_{\lambda_\rho}\big(\lambda_\rho^{-\frac{2}{(N-2)(2^*-q)}} x\big),$$
	where $u_{\lambda_\rho}$ is the ground state of $(P_{\lambda_\rho})$ and 
	\begin{equation}\label{lrho}
		\lambda_\rho\sim\left\{
		\begin{array}{lcl}
			\rho^{\frac{(N-2)^2(q-2)(2^*-q)}{8}}&\text{if}&N\ge 5,\smallskip\\
			\rho^{\frac{(q-2)(4-q)}{2}}\left(W_0\Big(\frac4{(4-q)^2}\rho^{-\frac{2(q-2)}{4-q}}\Big)\right)^{\frac14(4-q)^2}&\text{if}&N=4,\smallskip\\
			\rho^{\frac{(q-4)(6-q)}{q-2}}&\text{if}&N=3.
		\end{array}\right.
	\end{equation}
here $W_0(\cdot)$ is the principal branch of the Lambert $W$--function.\footnote{$W_0(x)$ is defined as the the unique real solution of the equation $ye^y=x$, $x\ge 0$.} In particular, as $\rho\to 0$, the ground states $v_\rho$ converge to a ground state of the critical Emden--Fowler equation \eqref{e12}, after the rescalings described in Theorem \ref{t1}.
\end{theorem}

\proof
Given $\rho>0$, assume that $v_\rho\in H^1(\R^N)$ is a critical point of $J$ on $S_\rho$ with a critical level $m_\rho=J(v_\rho)$ and with a Lagrange multiplier $\omega_\rho\in \R$. 
Denote
$$A=\|\nabla v_\rho\|_2^2,\quad B=\|v_\rho\|_q^q,\quad C=\|v_\rho\|_{2^*}^{2^*}.$$
Applying Nehari and Poho\v zaev identities (cf. \cite{Berestycki-1}), we obtain the system
\begin{equation}\label{abc-syst}
	\left\{
	\begin{aligned}
		\frac12A-\frac1q B-\frac1{2^*} C&=m_\rho\\
		A-B-C&=-\omega_\rho\rho^2\\
		\frac{N-2}{2}A-\frac{N}{q}B-\frac{N}{2^*}C&=-\frac{N}{2} \omega_\rho\rho^2.
	\end{aligned}
	\right.
\end{equation}
This is a linear system and the determinant is zero when $q=2^*$. 
We solve the system explicitly to obtain
\begin{equation}\label{e-ABC}
	\omega_\rho=\frac{(N-2)(2^*-q)}{2q\rho^2}B,\quad m_\rho=\frac{1}{N}A-\frac{N}{2}\Big(\frac{1}{q}-\frac{1}{2^*}\Big)B,\quad C=A-N\Big(\frac{1}{2}-\frac{1}{q}\Big)B.
\end{equation}
From the first relation we can deduce
\begin{equation}\label{e-r1}
	\rho^2\omega_\rho=\frac{(N-2)(2^*-q)}{2q}B>0.
\end{equation}
Taking into account the rescaling \eqref{e-olambda}, we obtain
\begin{equation}\label{e-o1}
	B=\|v_\rho\|_q^q=\lambda_\rho^{-\frac{q}{p-q}}\lambda_\rho^{\frac{p-2}{2(p-q)}N}\|u_{\lambda_\rho}\|_q^q=\lambda_\rho\|u_{\lambda_\rho}\|_q^q,
\end{equation}
and from \eqref{e-r1} we have
\begin{equation}\label{e-o2}
	\rho^2\lambda_\rho^{-\frac{4}{(N-2)(2^*-q)}}=c\lambda_\rho\|u_{\lambda_\rho}\|_q^q,
\end{equation}
or 
\begin{equation}\label{e-o3}
	\rho^2=
	c\lambda_\rho^{1+\frac{4}{(N-2)(2^*-q)}}\|u_{\lambda_\rho}\|_q^q.
\end{equation}
Recall that according to Theorem \ref{t1}, for small $\lambda>0$ the $L^q$--norm of ground states of $(P_\lambda)$ satisfies
\begin{equation}
	\|u_\lambda\|_q^q\sim\left\{
	\begin{array}{lcl}
		\lambda^{\frac{2^*-q}{q-2}}&\text{if}&N\ge 5,\smallskip\\
		\lambda^{\frac{4-q}{q-2}}(\ln\frac{1}{\lambda})^{-\frac{4-q}{q-2}}&\text{if}&N=4,\smallskip\\
		\lambda^{\frac{6-q}{q-4}}&\text{if}&N=3.
	\end{array}\right.
\end{equation}
Substituting into \eqref{e-o3} we obtain
\begin{equation}\label{rhol}
	\rho\sim\left\{
	\begin{array}{lcl}
		\lambda_\rho^{\frac{8}{(N-2)^2(q-2)(2^*-q)}}&\text{if}&N\ge 5,\smallskip\\
		\lambda_\rho^{\frac{2}{(q-2)(4-q)}}(\ln\frac{1}{\lambda})^{-\frac{4-q}{2(q-2)}}&\text{if}&N=4,\smallskip\\
		\lambda_\rho^{\frac{q-2}{(q-4)(6-q)}}&\text{if}&N=3,
	\end{array}\right.
\end{equation}
and then \eqref{lrho} follows after the inversion.
\qed

\begin{remark}
We conjecture that the estimates \eqref{lrho} remain valid beyond the uniqueness scenario of \cite[Theorem 1]{Pucci}. The proof of this would require a direct analysis of the ground states of $J$ on $S_\rho$ adapting the techniques in this paper, and thus bypassing the unconstrained problem $(P_\lambda)$.
Note that the estimate \eqref{lrho} is different from the estimates in \cite[Proposition 4.1, 4.2]{Wei-1}, where $\rho$ is fixed.
\end{remark}

\section{Rescalings and preliminary estimates as $\lambda\to 0$}\label{s3}

The formal limit equation for $(P_\lambda)$ as $\lambda\to 0$ is given by
\[
-\Delta u+u=u^{2^*-1} \quad \text{in $\R^N$.}
\eqno(P_0)
\]
Recall that $(P_0)$ has no nontrivial solutions in $H^1(\R^N)$, this follows from Poho\v zaev's identity.
We denote the Nehari manifolds for $(P_\lambda)$ and $(P_0)$ as follows:
$$
\mathcal M_\lambda:=\left\{ u\in H^1(\mathbb R^N)\setminus\{0\}  \ \left | \ \int_{\mathbb R^N}|\nabla u|^2+|u|^2=\int_{\mathbb R^N}|u|^{2^*}+\lambda |u|^q \right. \right\}.
$$
$$
\mathcal M_0:=\left\{ u\in H^1(\mathbb R^N)\setminus\{0\}  \ \left | \ \int_{\mathbb R^N}|\nabla u|^2+|u|^2=\int_{\mathbb R^N}|u|^{2^*} \right. \right\}.
$$
Denote
$$
I_0(u):=\frac{1}{2}\int_{\mathbb R^N}\left(|\nabla u|^2+|u|^2\right)-\frac{1}{p}\int_{\mathbb R^N}|u|^p
$$
the limiting energy functional $I_0:  H^1(\mathbb  R^N)\to \mathbb R$.
It is easy to see  that 
$$
m_\lambda^*:=\inf_{u\in \mathcal M_\lambda}I_\lambda(u), \qquad 
m_0^*:=\inf_{u\in \mathcal M_0}I_0(u).
$$
are well defined and positive. 
Let $u_\lambda$ be the ground state for $(P_\lambda)$ constructed in 
Theorem \ref{t0}. 
Then we have the following

\begin{lemma}\label{l21}
The family of solutions $\{u_\lambda\}_{\lambda>0}$ is bounded in $H^1(\mathbb R^N)$.
\end{lemma}

\begin{proof}
It is not hard to show that $m_\lambda^*\le m_0^*$. Moreover, we have 
\begin{align*}
m_\lambda^*&=I_\lambda(u_\lambda)=I_\lambda(u_\lambda)-\frac{1}{q}I'_\lambda(u_\lambda)u_\lambda\\
&=\left(\frac{1}{2}-\frac{1}{q}\right)\int_{\mathbb R^N}|\nabla u_\lambda|^2+|u_\lambda|^2+\left(\frac{1}{q}-\frac{1}{2^*}\right)\int_{\mathbb R^N}|u_\lambda|^{2^*}\\
&\ge\left(\frac{1}{2}-\frac{1}{q}\right)\int_{\mathbb R^N}|\nabla u_\lambda|^2+|u_\lambda|^2.
\end{align*}
Therefore, $\{u_\lambda\}$ is bounded in $H^1(\mathbb R^N)$.
\end{proof}

For $\lambda>0$, define the rescaling
\begin{equation}\label{e22}
v(x)=\lambda^\frac{1}{q-2} u\big(\lambda^\frac{2^*-2}{2(q-2)}x\big). 
\end{equation}
Rescaling \eqref{e22} transforms $(P_\lambda)$ into the equivalent equaition 
\[
-\Delta v+\lambda^\sigma v=v^{2^*-1}+\lambda^\sigma v^{q-1} 
\quad\text{in} \ \R^N,
 \eqno(Q_\lambda)
\]
where  
\begin{equation}\label{e23}
\sigma:=\frac{2^*-2}{q-2}=\frac{4}{(N-2)(q-2)}.
\end{equation}
The corresponding energy functional is given by 
\begin{equation}\label{e24}
J_\lambda(v)=\frac{1}{2}\int_{\mathbb R^N}|\nabla v|^2+\lambda^\sigma|v|^2-\frac{1}{2^*}\int_{\mathbb R^N}|v|^{2^*}-\frac{1}{q}\lambda^\sigma\int_{\mathbb R^N}|v|^q.
\end{equation}
The formal limit equation for $(Q_\lambda)$ as $\lambda\to 0$ is given by the critical Emden--Fowler equation
\[
-\Delta v=v^{2^*-1} \quad \text{in $\R^N$.}
\eqno(Q_0)
\]
We denote their corresponding Nehari manifolds as follows:
$$
\mathcal N_\lambda:=\left\{ v\in H^1(\mathbb R^N)\setminus\{0\}  \ \left | \ \int_{\mathbb R^N}|\nabla v|^2+\lambda^\sigma |v|^2=\int_{\mathbb R^N}|v|^{2^*}+\lambda^\sigma |v|^q \right. \right\}.
$$
$$
\mathcal{N}_0:=
\left\{v\in D^{1,2}(\mathbb R^N)\setminus\{0\} \ \left | \ \int_{\mathbb R^N}|\nabla v|^2=\int_{\mathbb R^N}|v|^{2^*}\  \right. \right\}. 
$$
Then
$$
m_\lambda:=\inf_{v\in \mathcal {N}_\lambda}J_\lambda(v), \qquad  m_0:=\inf_{v\in \mathcal {N}_0}J_0(v) 
$$
are well-defined.  It is well known that  $m_0$ is attained on $\mathcal N_0$ by the Talenti function
$$
U_1(x):=[N(N-2)]^{\frac{N-2}{4}}\left(\frac{1}{1+|x|^2}\right)^{\frac{N-2}{2}}
$$
and the family of its rescalings 
\begin{equation}\label{e24+}
U_\rho(x):=\rho^{-\frac{N-2}{2}}U_1(x/\rho),  \quad \rho>0.
\end{equation}
For $v\in H^1(\mathbb R^N)\setminus \{0\}$, we set
\begin{equation}\label{e25}
\tau(v):=\frac{\int_{\mathbb R^N}|\nabla v|^2}{\int_{\mathbb R^N}|v|^{2^*}}.
\end{equation}
Then $(\tau(v))^{\frac{N-2}{4}}v\in \mathcal N_0$ for any $v\in H^1(\mathbb R^N)\setminus\{0\}$,  and $v\in \mathcal N_0$  if and only if $\tau(v)=1$.

It is standard to verify the following.

\begin{lemma}\label{l22}
Let $\lambda>0$, $u\in H^1(\mathbb R^N)$ and $v$ is the rescaling \eqref{e22} of $u$.  Then: 

\begin{enumerate}
\item[$(a)$] $\|\nabla v\|_{2}^{2}=\|\nabla u\|_{2}^{2}$, $\|v\|_{2^*}^{2^*}=\|u\|_{2^*}^{2^*}$,\medskip

\item[$(b)$] $\lambda^{\sigma} \|v\|_2^2=\| u\|_2^2$, $\lambda^\sigma\|v\|_q^q=\lambda\|u\|_q^q$,\medskip

\item[$(c)$] $J_\lambda(v)=I_\lambda(u)$, $m_\lambda=m_\lambda^*$.
\end{enumerate}
\end{lemma}

In particular, if $v_\lambda$ is the rescaling \eqref{e22} of the ground state $u_\lambda$, then $J_\lambda(v_\lambda)=I_\lambda(u_\lambda)$ and hence  $v_\lambda$ is the ground state of $(Q_\lambda)$. 
Moreover, $v_\lambda$ satisfies the Poho\v zaev's identity \cite{Berestycki-1}:
\begin{equation}\label{e26}
\frac{1}{2^*}\int_{\mathbb R^N}|\nabla v_\lambda|^2+\frac{\lambda^\sigma}{2}\int_{\mathbb R^N}|v_\lambda|^2=\frac{1}{2^*}\int_{\mathbb R^N}|v_\lambda|^{2^*}+\frac{\lambda^\sigma}{q}\int_{\mathbb R^N}|v_\lambda|^q.
\end{equation}
Define the Poho\v zaev manifold 
$$
\mathcal P_\lambda:=\{v\in H^1(\mathbb R^N)\setminus\{0\} \   | \  P_\lambda(v)=0 \},
$$
where 
\begin{equation}\label{e27}
P_\lambda(v):=\frac{N-2}{2}\int_{\mathbb R^N}|\nabla v|^2+\frac{\lambda^\sigma N}{2}\int_{\mathbb R^N}|v|^2-\frac{N}{2^*}\int_{\mathbb R^N}|v|^{2^*}-\frac{\lambda^\sigma N}{q}\int_{\mathbb R^N}|v|^q.
\end{equation}
Clearly, $v_\lambda\in \mathcal P_\lambda$. Moreover,  we have the following minimax characterizations for the least energy level $m_\lambda$.

\begin{lemma}\label{l23}
Let $\lambda\ge 0$. Set
$$
v_t(x)=\left\{\begin{array}{ccl} v(\frac{x}{t}) &\text{if}&  t>0,\smallskip\\
0 &\text{if}& t=0.
\end{array}\right.
$$
Then 
$$m_\lambda=\inf_{v\in H^1(\mathbb R^N)\setminus\{0\}}\sup_{t\ge 0}J_\lambda(tv)=\inf_{v\in H^1(\mathbb R^N)\setminus\{0\}}\sup_{t\ge 0}J_\lambda(v_t).
$$
In particular, we have $m_\lambda=J_\lambda(v_\lambda)=\sup_{t>0}J_\lambda(tv_\lambda)=\sup_{t>0}J_\lambda((v_\lambda)_t)$. 
%A similar result also holds for $m_0$ and $J_0$.
\end{lemma}

\begin{proof} 
The proof is standard and thus omitted.  We refer the reader to \cite[Theorem 1.1]{Li-1}, or to \cite{Jeanjean-1}.
\end{proof}

\begin{lemma}\label{l24}
Let $\lambda>0$.
The rescaled family of ground states $\{v_\lambda\}$ is bounded in $H^1(\mathbb R^N)$. In particular, $\{v_\lambda\}$ is bounded in $L^p(\mathbb R^N)$ uniformly for all $p\in [2,2^*]$.
\end{lemma}

\begin{proof}
Since $\|\nabla v_\lambda\|_2=\|\nabla u_\lambda\|_2$ is bounded by Lemma \ref{l21} and Lemma \ref{l22}, we need only to show that $v_\lambda$ is bounded in $L^2(\mathbb R^N)$.  Since $v_\lambda\in \mathcal N_\lambda\cap \mathcal P_\lambda$, we have 
$$
\int_{\mathbb R^N}|\nabla v_\lambda|^2+\lambda^\sigma\int_{\mathbb R^N}|v_\lambda|^2-\int_{\mathbb R^N}|v_\lambda|^{2^*}-\lambda^\sigma\int_{\mathbb R^N}|v_\lambda|^q=0,
$$
and 
$$
\frac{1}{2^*}\int_{\mathbb R^N}|\nabla v_\lambda|^2+\frac{\lambda^\sigma}{2}\int_{\mathbb R^N}|v_\lambda|^2-\frac{1}{2^*}\int_{\mathbb R^N}|v_\lambda|^{2^*}-\frac{\lambda^\sigma}{q}\int_{\mathbb R^N}|v_\lambda|^q=0.
$$
It then follows that 
$$
\left(\frac{1}{2}-\frac{1}{2^*}\right)\lambda^\sigma\int_{\mathbb R^N}|v_\lambda|^2=\left(\frac{1}{q}-\frac{1}{2^*}\right)\lambda^\sigma\int_{\mathbb R^N}|v_\lambda|^q.
$$
Thus, we obtain
\begin{equation}\label{e28}
\int_{\mathbb R^N}|v_\lambda|^2=\frac{2(2^*-q)}{q(2^*-2)}\int_{\mathbb R^N}|v_\lambda|^q.
\end{equation}
By the Sobolev embedding theorem  and the  interpolation inequality, we obtain 
$$
\int_{\mathbb R^N}|v_\lambda|^q\le \left(\int_{\mathbb R^N}|v_\lambda|^2\right)^{\frac{2^*-q}{2^*-2}}\left(\int_{\mathbb R^N}|v_\lambda|^{2^*}\right)^{\frac{q-2}{2^*-2}}
\le \left(\int_{\mathbb R^N}|v_\lambda|^2\right)^{\frac{2^*-q}{2^*-2}}\left(\frac{1}{S}\int_{\mathbb R^N}|\nabla v_\lambda|^2\right)^{\frac{2^*(q-2)}{2(2^*-2)}},
$$
where $S$ is the best Sobolev constant.  Therefore,  we have 
$$
\left(\int_{\mathbb R^N}|v_\lambda|^2\right)^{\frac{q-2}{2^*-2}}\le \frac{2(2^*-q)}{q(2^*-2)}\left(\frac{1}{S}\int_{\mathbb R^N}|\nabla v_\lambda|^2\right)^{\frac{2^*(q-2)}{2(2^*-2)}}.
$$
It then follows from Lemma \ref{l22} that
\begin{equation}\label{e29}
\int_{\mathbb R^N}|v_\lambda|^2\le \left(\frac{2(2^*-q)}{q(2^*-2)}\right)^{\frac{2^*-2}{q-2}}\left(\frac{1}{S}\int_{\mathbb R^N}|\nabla u_\lambda|^2\right)^{2^*/2},
\end{equation}
which together with the boundedness of $u_\lambda$ in $H^1(\mathbb R^N)$ implies that $v_\lambda$ is bounded in $L^2(\mathbb R^N)$.  

Finally, for any $p\in [2,2^*]$, by \eqref{e29} and the interpolation inequality, we have 
$$
\int_{\mathbb R^N}|v_\lambda|^p\le \left(\int_{\mathbb R^N}|v_\lambda|^2\right)^{\frac{2^*-p}{2^*-2}}\left(\int_{\mathbb R^N}|v_\lambda|^{2^*}\right)^{\frac{p-2}{2^*-2}}\le \left(\frac{2(2^*-q)}{q(2^*-2)}\right)^{\frac{2^*-p}{q-2}}\left(\frac{1}{S}\int_{\mathbb R^N}|\nabla u_\lambda|^2\right)^{2^*/2},
$$
and 
$$
\lim_{q\to 2}\left(\frac{2(2^*-q)}{q(2^*-2)}\right)^{\frac{2^*-p}{q-2}}=e^{-N(2^*-p)/4},  \quad {\rm for \ any } \  p\in [2,2^*].
$$
Therefore, by Lemma \ref{l21}, $\{v_\lambda\}$ is bounded  in $L^p(\mathbb R^N)$  uniformly for $p\in [2,2^*]$. 
\end{proof} 

\begin{remark}\label{r21}
A straightforward computation shows that
$$
\lim_{q\to 2}\left(\frac{2}{q}\right)^{\frac{2^*-2}{q-2}}=e^{-\frac{2}{N-2}}, \qquad 
\lim_{q\to 2}\left(\frac{2^*-q}{2^*-2}\right)^{\frac{2^*-2}{q-2}}=e^{-1}
$$
and 
$$
  \lim_{q\to 2^*}\frac{1}{2^*-q}\left(\frac{2^*-q}{2^*-2}\right)^{\frac{2^*-2}{q-2}}=\frac{N-2}{4}.
$$
Therefore, we have 
$$
\left(\frac{2(2^*-q)}{q(2^*-2)}\right)^{\frac{2^*-2}{q-2}}\sim 2^*-q.
$$
\end{remark}

Next we obtain an estimation of the least energy.

\begin{lemma}\label{l26}
Let 
\begin{equation}\label{e210}
Q(q):= \left(\frac{2^*-q}{2^*-2}\right)^{\frac{2^*-q}{q-2}}  \quad      and     \qquad    G(q):=\frac{q-2}{2^*-2}Q(q).
\end{equation}
Then $Q(q)\sim 1$, $G(q)\sim q-2$ 
and for all $\lambda>0$:
\smallskip

$(i)$  \   $1<\tau(v_\lambda)\le 1+G(q)\lambda^\sigma$, 
%for any ground states $v_\lambda\in H^1(\mathbb R^N)\setminus\{0\}$. 

$(ii)$  \   $m_0>m_\lambda> m_0\left(1-\lambda^\sigma NG(q)(1+G(q)\lambda^\sigma)^{\frac{N-2}{2}}\right)$.
\end{lemma}

\begin{proof}
For $\theta\in (0,1)$, consider the function 
$$
g(x):=x^\theta\left(1-x^{1-\theta}\right), \qquad x\in [0, +\infty).
$$
It is easy to see that 
$$
\max_{x\ge 0}g(x)= \theta^{\frac{\theta}{1-\theta}}(1-\theta).
$$
Using the interpolation inequality,
$$
\int_{\mathbb R^N}|v_\lambda|^q\le \left(\int_{\mathbb R^N}|v_\lambda|^2\right)^{\frac{2^*-q}{2^*-2}}\left(\int_{\mathbb R^N}|v_\lambda|^{2^*}\right)^{\frac{q-2}{2^*-2}}, 
$$
we see that 
\begin{equation}\label{e210R}
\frac{\int_{\mathbb R^N}|v_\lambda|^q-\int_{\mathbb R^N}|v_\lambda|^2}{\int_{\mathbb R^N}|v_\lambda|^{2^*}}\le \zeta_\lambda^{\theta_q}(1-\zeta_\lambda^{1-\theta_q})\le \theta_q^{\frac{\theta_q}{1-\theta_q}}(1-\theta_q)=G(q), 
\end{equation}
where 
$$
\theta_q=\frac{2^*-q}{2^*-2}, \qquad \zeta_\lambda=\frac{\int_{\mathbb R^N}|v_\lambda|^2}{\int_{\mathbb R^N}|v_\lambda|^{2^*}}.
$$
Since $v_\lambda\in \mathcal N_\lambda$,  by \eqref{e28} and \eqref{e210R}, we have 
$$
1<\tau(v_\lambda)=\frac{\int_{\mathbb R^N}|\nabla v_\lambda|^2}{\int_{\mathbb R^N}|v_\lambda|^{2^*}}=1+\lambda^\sigma\frac{\int_{\mathbb R^N}|v_\lambda|^q-\int_{\mathbb R^N}|v_\lambda|^2}{\int_{\mathbb R^N}|v_\lambda|^{2^*}}\le 1+\lambda^\sigma G(q).
$$
This proves $(i)$.  To prove $(ii)$, we  first note that by \eqref{e28} and \eqref{e210R} the following inequality holds
$$
\frac{1}{q}\int_{\mathbb R^N}|v_\lambda|^q-\frac{1}{2}\int_{\mathbb R^N}|v_\lambda|^2\le \int_{\mathbb R^N}|v_\lambda|^q-\int_{\mathbb R^N}|v_\lambda|^2\le G(q)\int_{\mathbb R^N}|v_\lambda|^{2^*}.
$$
 Since $v_\lambda\in \mathcal{N}_\lambda$, by  \eqref{e28}, we also have
\begin{align}
m_\lambda&=(\frac{1}{2}-\frac{1}{2^*})\int_{\mathbb R^N}|\nabla v_\lambda|^2+(\frac{1}{2}-\frac{1}{2^*})\lambda^\sigma\int_{\mathbb R^N}|v_\lambda|^2-(\frac{1}{q}-\frac{1}{2^*})\lambda^\sigma\int_{\mathbb R^N}|v_\lambda|^q\nonumber\\
&=\frac{1}{N}\int_{\mathbb R^N}|\nabla v_\lambda|^2.\nonumber
\end{align}
Therefore, by Lemma \ref{l23} and the definition of $\tau(v_\lambda)$, we find
\begin{align}\label{e210RR}
m_0&\le \sup_{t\ge 0} J_\lambda((v_\lambda)_t)+\lambda^\sigma (\tau(v_\lambda))^{N/2}\left[\frac{1}{q}\int_{\mathbb R^N}|v_\lambda|^q-\frac{1}{2}\int_{\mathbb R^N}|v_\lambda|^2\right]\\
&\le m_\lambda +\lambda^\sigma(\tau(v_\lambda))^{\frac{N}{2}}\int_{\mathbb R^N}| v_\lambda|^{2^*}G(q)\nonumber\\
&\le m_\lambda+\lambda^\sigma(\tau(v_\lambda))^{\frac{N-2}{2}}\int_{\mathbb R^N}|\nabla v_\lambda|^2G(q)\nonumber\\
&\le m_\lambda\left[1+\lambda^\sigma NG(q)(1+G(q)\lambda^\sigma)^\frac{N-2}{2}\right].\nonumber
\end{align}
Hence, we obtain
$$
m_\lambda\ge \frac{m_0}{1+\lambda^\sigma NG(q)(1+G(q)\lambda^\sigma)^\frac{N-2}{2}}>m_0\left[ 1-\lambda^\sigma NG(q)(1+G(q)\lambda^\sigma)^\frac{N-2}{2}\right],
$$
which completes the proof. 
\end{proof}

\begin{lemma}\label{l27}
Assume $N\ge 5$.  Then there exists a constant $c_0>0$, which is independent of $q$, $\lambda$, and such that for all small $\lambda>0$,
$$m_\lambda\le m_0-\lambda^\sigma \left\{\frac{c_0}{q}\left(\frac{2}{q}\right)^{\frac{2^*-q}{q-2}}G(q)
      -\lambda^\sigma\frac{2Nm_0}{q-2}G(q)^2\right\}.
$$
 \end{lemma}   
\begin{proof}
For each $\rho>0$, the family  $\{U_\rho\}$ of radial ground states of $(Q_0)$ defined in \eqref{e24+} verifies
\begin{equation}\label{e211}
 \|U_\rho\|_2^2=\rho^2\|U_1\|_2^2, \qquad \|U_\rho\|_q^q=\rho^{\frac{2(2^*-q)}{2^*-2}}\|U_1\|_q^q.
\end{equation}
 Let $g_0(\rho)=\frac{1}{q}\int_{\mathbb R^N}|U_\rho|^q-\frac{1}{2}\int_{\mathbb R^N}|U_\rho|^2$. Then there exists a unique $\rho_0=\rho_0(q)\in (0,+\infty)$ given by 
 $$
 \rho_0=\left(\frac{2(2^*-q)}{q(2^*-2)}\cdot \frac{\|U_1\|_q^q}{\|U_1\|_2^2}\right)^{\frac{2^*-2}{2(q-2)}}, 
 $$
 such that 
 \begin{equation}\label{e212}
 g_0(\rho_0)=\sup_{\rho>0}g_0(\rho)=\frac{1}{q}\left(\frac{2}{q}\right)^{\frac{2^*-q}{q-2}}G(q)\left(\frac{\|U_1\|_q^{q(2^*-2)}}{\|U_1\|_2^{2(2^*-q)}}\right)^{\frac{1}{q-2}}.
 \end{equation}
Since $N\ge 5$, by using the Lebesgue Dominated Convergence Theorem, it is not hard to show that
$$
\lim_{q\to 2}\left(\frac{\|U_1\|_q^{q(2^*-2)}}{\|U_1\|_2^{2(2^*-q)}}\right)^{\frac{1}{q-2}}=\exp\left(\frac{2\int_0^\infty\kappa(r)\ln\frac{1}{1+r^2}dr}
{\int_0^\infty\kappa(r)dr}\right)\cdot
\int_0^\infty\kappa(r)dr, 
$$
where $\kappa(r)=(1+r^2)^{2-N}r^{N-1}$. Therefore, we conclude that
\begin{equation}\label{e213}
c_0:=\inf_{q\in (2, 2^*)}\left(\frac{\|U_1\|_q^{q(2^*-2)}}{\|U_1\|_2^{2(2^*-q)}}\right)^{\frac{1}{q-2}}>0.
\end{equation}
Thus, we get
$$
 g_0(\rho_0)\ge 
 \frac{c_0}{q}\left(\frac{2}{q}\right)^{\frac{2^*-q}{q-2}}G(q). 
  $$  
Put $U_0(x):=U_{\rho_0}(x)$, then by Lemma \ref{l23}, we have 
\begin{align}\label{e214}
m_\lambda&\le \sup_{t\ge 0}J_\lambda(tU_0)=J_\lambda(t_\lambda U_0)\\
&=\frac{t^2_\lambda}{2}\int_{\mathbb R^N}|\nabla U_0|^2-\frac{t^{2^*}_\lambda}{2^*}\int_{\mathbb R^N}|U_0|^{2^*}+\lambda^\sigma\int_{\mathbb R^N}\frac{t^2_\lambda}{2}|U_0|^2-\frac{t^q_\lambda}{q}|U_0|^q\nonumber\\
&\le \sup_{t\ge 0}\left(\frac{t^2}{2}-\frac{t^{2^*}}{2^*}\right)\int_{\mathbb R^N}|\nabla U_0|^2 +\lambda^\sigma
\int_{\mathbb R^N}\frac{t_\lambda^2}{2}|U_0|^2-\frac{t_\lambda^q}{q}|U_0|^q\nonumber\\
&= m_0+\lambda^\sigma\int_{\mathbb R^N}\frac{t_\lambda^2}{2}|U_0|^2-\frac{t_\lambda^q}{q}|U_0|^q.\nonumber
\end{align}
It follows from $\frac{d}{dt}J_\lambda(tU_0)\big|_{t=t_\lambda}=0$ and $\int_{\mathbb R^N}|\nabla U_0|^2=\int_{\mathbb R^N}|U_0|^{2^*}=Nm_0$ that
$$
Nm_0+\lambda^\sigma \int_{\mathbb R^N}|U_0|^2=t_\lambda^{2^*-2}Nm_0+t_\lambda^{q-2}\lambda^\sigma\int_{\mathbb R^N}|U_0|^q.
$$
Recall that  $g_0(\rho_0)=\frac{1}{q}\int_{\mathbb R^N}|U_0|^q-\frac{1}{2}\int_{\mathbb R^N}|U_0|^2>0$, it follows that $\int_{\mathbb R^N}|U_0|^q>\int_{\mathbb R^N}|U_0|^2.$
If $t_\lambda\ge 1$, then
$$
Nm_0+\lambda^\sigma \int_{\mathbb R^N}|U_0|^2\ge t_\lambda^{q-2}\left\{Nm_0+\lambda^\sigma\int_{\mathbb R^N}|U_0|^q\right\}
$$
and hence 
$$
t_\lambda\le \left(\frac{Nm_0+\lambda^\sigma\int_{\mathbb R^N}|U_0|^2}{Nm_0+\lambda^\sigma\int_{\mathbb R^N}|U_0|^q}\right)^{\frac{1}{q-2}}<1,
$$
a contradiction. Therefore, $t_\lambda<1$ and hence
$$
Nm_0+\lambda^\sigma \int_{\mathbb R^N}|U_0|^2< t_\lambda^{q-2}\left\{Nm_0+\lambda^\sigma\int_{\mathbb R^N}|U_0|^q\right\},
$$
from which it follows that
\begin{equation}\label{e215}
\left(\frac{Nm_0+\lambda^\sigma\int_{\mathbb R^N}|U_0|^2}{Nm_0+\lambda^\sigma\int_{\mathbb R^N}|U_0|^q}\right)^{\frac{1}{q-2}}<t_\lambda
<1.
%\left(\frac{Nm_0+\lambda^\sigma\int_{\mathbb R^N}|U_0|^2}{Nm_0+\lambda^\sigma\int_{\mathbb R^N}|U_0|^q}\right)^{\frac{1}{2^*-2}}.
\end{equation}
Let 
$$A_\lambda:=\frac{\int_{\mathbb R^N}|U_0|^q-\int_{\mathbb R^N} |U_0|^2}{Nm_0+\lambda^\sigma\int_{\mathbb R^N}|U_0|^q}.$$ Then $A_\lambda\le \frac{1}{Nm_0}[\int_{\mathbb R^N}|U_0|^q-\int_{\mathbb R^N} |U_0|^2]$ and 
\begin{equation}\label{e215R}
[1-\lambda^\sigma A_\lambda]^{\frac{1}{q-2}}<t_\lambda<1.
%[1-\lambda^\sigma A_\lambda]^{\frac{1}{2^*-2}}.
\end{equation}
Let $g(t):=\frac{t^2}{2}\int_{\mathbb R^N}|U_0|^2-\frac{t^q}{q}\int_{\mathbb R^N}|U_0|^q$, and $h(x):=g([1-x]^{\frac{1}{q-2}})$ for $x\in [0,1]$.
Then  $g(t)$ has an unique miximum point at $t_0=\left(\frac{\int_{\mathbb R^N}|U_0|^2}{\int_{\mathbb R^N}|U_0|^q}\right)^{\frac{1}{q-2}}$ and is strictly decreasing in $(t_0,1)$, and for small $x>0$, we have 
$$
h'(x)=\frac{1}{q-2}[1-x]^{\frac{q-4}{q-2}}\left[-\int_{\mathbb R^N}|U_0|^2+(1-x)\int_{\mathbb R^N}|U_0|^q\right]>0.
$$
Therefore, for small $\lambda>0$, it follows from \eqref{e215R} and the monotonicity of $g(t)$ in $(t_0,1)$ that
$$
g(t_\lambda)\le g([1-\lambda^\sigma A_\lambda]^\frac{1}{q-2})=h(\lambda^\sigma A_\lambda)=\frac{1}{2}\int_{\mathbb R^N}|U_0|^2-\frac{1}{q}\int_{\mathbb R^N}|U_0|^q+h'(\xi)\lambda^\sigma A_\lambda, 
$$
for some $\xi\in (0, \lambda^\sigma A_\lambda)$.
Since for small $\lambda>0$, we have 
$$
h'(\xi)\le \frac{2}{q-2}\left[\int_{\mathbb R^N}|U_0|^q-\int_{\mathbb R^N}|U_0|^2\right], 
$$
and similar to \eqref{e210R}, we have 
$$
\frac{\int_{\mathbb R^N}|U_0|^q-\int_{\mathbb R^N}|U_0|^2}{\int_{\mathbb R^N}|U_0|^{2^*}}\le G(q),
$$
thus,  by the definition of $A_\lambda$, we obtain that 
\begin{align*}
g(t_\lambda)&\le\frac{1}{2}\int_{\mathbb R^N}|U_0|^2-\frac{1}{q}\int_{\mathbb R^N}|U_0|^q+\frac{2\lambda^\sigma}{Nm_0(q-2)}\left[\int_{\mathbb R^N}|U_0|^q-\int_{\mathbb R^N}|U_0|^2\right]^2\\
&=-g_0(\rho_0)+\frac{2\lambda^\sigma}{Nm_0(q-2)}\left[Nm_0\frac{\int_{\mathbb R^N}|U_0|^q-\int_{\mathbb R^N}|U_0|^2}{\int_{\mathbb R^N}|U_0|^{2^*}}\right]^2\\
&\le-g_0(\rho_0)+\lambda^\sigma\frac{2Nm_0}{q-2}G(q)^2,
\end{align*}
from which the conclusion follows.
\end{proof}

\begin{lemma}\label{l28}
There exists a constant $\varpi=\varpi(q)>0$ such that  for all small $\lambda>0$, 
$$
m_\lambda\le\left\{\begin{array}{lclcl} m_0-\lambda^\sigma\left(\ln \frac{1}{\lambda}\right)^{-\frac{4-q}{q-2}}\varpi&=&m_0-\lambda^{\frac{2}{q-2}}(\ln\frac{1}{\lambda})^{-\frac{4-q}{q-2}}\varpi   &\text{if}&N=4,\smallskip\\
m_0-\lambda^{\sigma+\frac{2(6-q)}{(q-2)(q-4)}}\varpi&=&m_0-\lambda^{\frac{2}{q-4}}\varpi &\text{if}&N=3 \  \text{and} \  q>4.\end{array}\right.
$$
\end{lemma}

\begin{proof}
Let  $\rho>0$, $R\gg 1$ be a large parameter and $\eta_R\in C_0^\infty(\mathbb R)$ is a cut-off function such that $\eta_R(r)=1$ for $|r|<R$, $0<\eta_R(r)<1$ for $R<|r|<2R$, $\eta_R(r)=0$ for $|r|>2R$ and $|\eta'_R(r)\le 2/R$.  

For $\ell\gg 1$, a straightforward computation shows that 
$$
\int_{\mathbb R^N}|\nabla (\eta_\ell U_1)|^2=
\left\{ 
\begin{array}{ccl} Nm_0+O(\ell^{-2}) &\text{if}& N=4,\\
Nm_0+O(\ell^{-1})   &\text{if}& N=3.
\end{array}
\right.
$$ 
$$
\int_{\mathbb R^N} |\eta_\ell U_1|^{2^*}=Nm_0+O(\ell^{-N}),
$$
$$
\int_{\mathbb R^N}|\eta_\ell U_1|^2=\left\{ 
\begin{array}{ccl} \ln \ell(1+o(1)) &\text{if}& N=4,\\
\ell(1+o(1)) &\text{if}& N=3.
\end{array}\right.
$$ 
By Lemma \ref{l23}, we find
\begin{align}\label{l216}
m_\lambda&\le \sup_{t\ge 0}J_\lambda((\eta_RU_\rho)_t)=J_\lambda((\eta_RU_\rho)_{t_\lambda})\\
&\le \sup_{t\ge 0}\left(\frac{t^{N-2}}{2}\int_{\mathbb R^N}|\nabla(\eta_RU_\rho)|^2-\frac{t^{N}}{2^*}\int_{\mathbb R^N}|\eta_RU_\rho|^{2^*}\right)\nonumber\\
&\quad  -\lambda^\sigma
t_\lambda^N\left[\int_{\mathbb R^N}\frac{1}{q}|\eta_RU_\rho|^q-\frac{1}{2}|\eta_RU_\rho|^2\right]\nonumber\\
&=  (I)-\lambda^\sigma (II).\nonumber
\end{align}
where 
\begin{equation}\label{l217}
t_\lambda
=\left(\frac{(N-2)\int_{\mathbb R^N}|\nabla(\eta_RU_\rho)|^2}{2N\left[\frac{1}{2^*}\int_{\mathbb R^N}|\eta_RU_\rho|^{2^*}-\frac{\lambda^\sigma}{2}\int_{\mathbb R^N}|\eta_RU_\rho|^2+\frac{\lambda^\sigma}{q}\int_{\mathbb R^N}|\eta_RU_\rho|^q\right]}\right)^{\frac{1}{2}}.
\end{equation}
Set $\ell=R/\rho$, then
\begin{equation}\label{l218}
(I)=\frac{1}{N}\frac{\|\nabla(\eta_\ell U_1\|_2^N}{\|\eta_\ell U_1\|_{2^*}^N}=
\left\{\begin{array}{rcl} m_0+O(\ell^{-2})&\text{if}&  N=4,\smallskip\\
m_0+O(\ell^{-1})&\text{if}&N=3.
\end{array}\right.
\end{equation}
Since 
$$
\varphi(\rho):=\int_{\mathbb R^N}\frac{1}{q}|\eta_RU_\rho|^q-\frac{1}{2}|\eta_RU_\rho|^2=\frac{1}{q}\rho^{N-\frac{N-2}{2}q}\int_{\mathbb R^N}|\eta_\ell U_1|^q-\frac{1}{2}\rho^2\int_{\mathbb R^N}|\eta_\ell U_1|^2
$$
takes its maximum value $\varphi(\rho_0)$ at the unique point $\rho_0>0$, and 
$$
\varphi(\rho_0)=\sup_{\rho\ge 0}\varphi(\rho)=\frac{1}{q}\left(\frac{2}{q}\right)^{\frac{2^*-q}{q-2}}G(q)\left(\frac{\|\eta_\ell U_1\|_q^{q(2^*-2)}}{\|\eta_\ell U_1\|_2^{2(2^*-q)}}\right)^{\frac{1}{q-2}}\le\frac{1}{q}\left(\frac{2}{q}\right)^{\frac{2^*-q}{q-2}}G(q)\| U_1\|_{2^*}^{2^*},
$$
where we have used the interpolation inequality
$$
\|\eta_\ell U_1\|_q^q\le \|\eta_\ell U_1\|_2^{\frac{2(2^*-q)}{2^*-2}}\|\eta_\ell U_1\|_{2^*}^{\frac{2^*(q-2)}{2^*-2}}.
$$
Then we obtain
\begin{align}\label{l219}
(II)&=\left(\frac{\|\nabla(\eta_\ell U_1)\|_2^2}{\|\eta_\ell U_1\|_{2^*}^{2^*}+\lambda^\sigma 2^*\varphi(\rho_0)}\right)^{N/2}\varphi(\rho_0)\\
&\ge \left(\frac{\|\nabla(\eta_\ell U_1)\|_2^2}{\|\eta_\ell U_1\|_{2^*}^{2^*}}\right)^{N/2}
\left[1-\lambda^\sigma\frac{N^2\varphi(\rho_0)}{(N-2)\|\eta_\ell U_1\||_{2^*}^{2^*}}\right]\varphi(\rho_0).\nonumber
\end{align}
Therefore, we have 
\begin{align}\label{l220}
m_\lambda&\le\frac{1}{N}\frac{\|\nabla(\eta_\ell U_1\|_2^N}{\|\eta_\ell U_1\|_{2^*}^N}\left\{ 1-\lambda^\sigma\frac{N}{\|\eta_\ell U_1\|_{2^*}^{(2^*-2)N/2}}\left[1-\lambda^\sigma\frac{N^2\varphi(\rho_0)}{(N-2)\|\eta_\ell U_1\||_{2^*}^{2^*}}\right]\varphi(\rho_0)\right\}\nonumber\\
&\le\frac{1}{N}\frac{\|\nabla(\eta_\ell U_1\|_2^N}{\|\eta_\ell U_1\|_{2^*}^N}\left\{ 1-\lambda^\sigma\frac{N}{2\|\eta_\ell U_1\|_{2^*}^{(2^*-2)N/2}}\varphi(\rho_0)\right\}\nonumber\\
&\le\frac{1}{N}\frac{\|\nabla(\eta_\ell U_1\|_2^N}{\|\eta_\ell U_1\|_{2^*}^N}\left\{ 1-\lambda^\sigma \frac{2}{m_0}\varphi(\rho_0)\right\}.\nonumber
\end{align}
For the rest of the proof, we consider separately the cases $N=4$ and $N=3$.

\smallskip
\paragraph{\sc Case $N=4$.} Since 
$$
\begin{array}{rcl}
\varphi(\rho_0)&=&\frac{1}{q}\left(\frac{2}{q}\right)^{\frac{2^*-q}{q-2}}G(q)\left(\frac{\| U_1\|_q^{q(2^*-2)}+o(1)}{[\ln\ell(1 +o(1))]^{2^*-q}}\right)^{\frac{1}{q-2}}\\
&=&
(\ln\ell)^{-\frac{2^*-q}{q-2}}\frac{1}{q}\left(\frac{2}{q}\right)^{\frac{2^*-q}{q-2}}G(q)\left(\| U_1\|_q^{q(2^*-2)}+o(1)\right)^{\frac{1}{q-2}}\\
&\ge& (\ln\ell)^{-\frac{2^*-q}{q-2}}\frac{1}{2q}\left(\frac{2}{q}\right)^{\frac{2^*-q}{q-2}}G(q)\| U_1\|_q^{\frac{q(2^*-2)}{q-2}},
\end{array}
$$
by \eqref{l218}, we have 
$$
m_\lambda\le [m_0+O(\ell^{-2})]\left\{ 1-\lambda^\sigma (\ln\ell)^{-\frac{2^*-q}{q-2}}\frac{1}{qm_0}\left(\frac{2}{q}\right)^{\frac{2^*-q}{q-2}}G(q)\| U_1\|_q^{\frac{q(2^*-2)}{q-2}}\right\}.
$$
Take $\ell=(1/\lambda)^M$. Then 
$$
m_\lambda\le [m_0+O(\lambda^{2M})]\left\{ 1-M^{-\frac{2^*-q}{q-2}}\lambda^\sigma (\ln\frac{1}{\lambda})^{-\frac{2^*-q}{q-2}}\frac{1}{qm_0}\left(\frac{2}{q}\right)^{\frac{2^*-q}{q-2}}G(q)\| U_1\|_q^{\frac{q(2^*-2)}{q-2}}\right\}.
$$
If $M>\frac{1}{q-2}$, then $2M>\sigma$, and hence
\begin{equation}\label{e221}
m_\lambda\le m_0-\lambda^\sigma (\ln\frac{1}{\lambda})^{-\frac{2^*-q}{q-2}}\frac{1}{2q}\left(\frac{2}{qM}\right)^{\frac{2^*-q}{q-2}}G(q)\| U_1\|_q^{\frac{q(2^*-2)}{q-2}}.
\end{equation}
Thus, if $N=4$, the result of Lemma \ref{l28} is proved by choosing 
$$
\varpi=\frac{1}{2q}\left(\frac{2}{qM}\right)^{\frac{2^*-q}{q-2}}G(q)\| U_1\|_q^{\frac{q(2^*-2)}{q-2}}.
$$

\paragraph{\sc Case $N=3$.} In this case, we always assume that $q\in (4,6)$.  Since 
\begin{align}
\varphi(\rho_0)&=\frac{1}{q}\left(\frac{2}{q}\right)^{\frac{2^*-q}{q-2}}G(q)\left(\frac{\| U_1\|_q^{q(2^*-2)}+o(1)}{[\ell(1 +o(1))]^{2^*-q}}\right)^{\frac{1}{q-2}}\\
&=
\ell^{-\frac{2^*-q}{q-2}}\frac{1}{q}\left(\frac{2}{q}\right)^{\frac{2^*-q}{q-2}}G(q)\left(\| U_1\|_q^{q(2^*-2)}+o(1)\right)^{\frac{1}{q-2}}\nonumber\\
&\ge \ell^{-\frac{2^*-q}{q-2}}\frac{1}{2q}\left(\frac{2}{q}\right)^{\frac{2^*-q}{q-2}}G(q)\| U_1\|_q^{\frac{q(2^*-2)}{q-2}},\nonumber
\end{align}
 we have 
$$
m_\lambda\le [m_0+O(\ell^{-1})]\left\{ 1-\lambda^\sigma \ell^{-\frac{2^*-q}{q-2}}\frac{1}{qm_0}\left(\frac{2}{q}\right)^{\frac{2^*-q}{q-2}}G(q)\| U_1\|_q^{\frac{q(2^*-2)}{q-2}}\right\}.
$$
Take $\ell=\delta^{-1}\lambda^{-\frac{2}{q-4}}$. Then 
$$
m_\lambda\le [m_0+\delta O(\lambda^{\frac{2}{q-4}})]\left\{ 1-\delta^{\frac{6-q}{q-2}}\lambda^{\frac{2}{q-4}}\frac{1}{qm_0}\left(\frac{2}{q}\right)^{\frac{6-q}{q-2}}G(q)\| U_1\|_q^{\frac{4q}{q-2}}\right\}.
$$
Since $\frac{6-q}{q-2}<1$, we can choose a small $\delta>0$ such that 
\begin{equation}\label{l222}
m_\lambda\le m_0-\lambda^{\frac{2}{q-4}}\frac{1}{2q}\left(\frac{2\delta}{q}\right)^{\frac{6-q}{q-2}}G(q)\| U_1\|_q^{\frac{4q}{q-2}},
\end{equation}
and take
$$
\varpi=\frac{1}{2q}\left(\frac{2\delta}{q}\right)^{\frac{6-q}{q-2}}G(q)\| U_1\|_q^{\frac{4q}{q-2}},
$$
which finished the proof in the case $N=3$.
\end{proof}

\begin{corollary}\label{c29}
Let $\delta_\lambda:=m_0-m_\lambda$, then 
$$
\lambda^\sigma\gtrsim \delta_\lambda\gtrsim 
\left\{\begin{array}{lcl} 
 \lambda^\sigma &\text{if}&  N\ge 5,\\
 \lambda^{\frac{2}{q-2}}(\ln\frac{1}{\lambda})^{-\frac{4-q}{q-2}}&\text{if}&N=4,\\
 \lambda^{\frac{2}{q-4}} &\text{if}& N=3 \ \text{and} \ q\in (4,6).
 \end{array}\right.
 $$
\end{corollary}

\begin{lemma}\label{l210}
Assume $N\ge 5$. Then for small $\lambda>0$,
\begin{equation}\label{e223}
\frac{2q}{2^*-2}Q(q)m_0\ge \|v_\lambda\|_q^q\ge Q(q)\left(\frac{c_0}{q}\left(\frac{2}{q}\right)^{\frac{2^*-q}{q-2}}-\lambda^\sigma 2NQ(q)m_0
\right)\frac{q(2^*-2)}{(\tau(v_\lambda))^{N/2}},
\end{equation}
where $c_0>0$ is given in Lemma \ref{l27}.   
In particular, 
$$
 \|v_\lambda\|_2^2\sim 2^*-q \quad and   \quad \|v_\lambda\|_q^q\sim 1.    
$$
\end{lemma}

\begin{proof}
Since 
$$
m_\lambda=\frac{1}{N}\int_{\mathbb R^N}|\nabla v_\lambda|^2=\frac{1}{N}\int_{\mathbb R^N}|v_\lambda|^{2^*}+\lambda^\sigma\frac{q-2}{2q}\int_{\mathbb R^N}|v_\lambda|^q,
$$
then by Lemma \ref{l26},  we get
$$
\lambda^\sigma\frac{q-2}{2q}\int_{\mathbb R^N}|v_\lambda|^q=\frac{\tau(v_\lambda)-1}{\tau(v_\lambda)}m_\lambda\le \lambda^\sigma G(q)m_0,
$$
and hence
$$
\|v_\lambda\|_q^q\le 2q\frac{G(q)}{q-2}m_0=\frac{2q}{2^*-2}Q(q)m_0.
$$
On the other hand, by \eqref{e28} and \eqref{e210RR}, we have
$$
m_0\le m_\lambda+\lambda^\sigma (\tau(v_\lambda))^{N/2}\frac{q-2}{q(2^*-2)}\int_{\mathbb R^N}|v_\lambda|^q.
$$
Therefore, it follows from Lemma \ref{l27} that 
$$
\|v_\lambda\|_q^q\ge \left(\frac{c_0}{q}\left(\frac{2}{q}\right)^{\frac{2^*-q}{q-2}}\frac{G(q)}{q-2}-\lambda^\sigma 2Nm_0\left(\frac{G(q)}{q-2}\right)^2\right)\frac{q(2^*-2)}{(\tau(v_\lambda))^{N/2}},
$$
from which the conclusion follows. 
 
A straightforward computation shows that
$$
\lim_{q\to 2}\left(\frac{2}{q}\right)^{\frac{2^*-q}{q-2}}=e^{-\frac{2}{N-2}}, \qquad 
\lim_{q\to 2}\left(\frac{2^*-q}{2^*-2}\right)^{\frac{2^*-q}{q-2}}=e^{-1}, \qquad  \lim_{q\to 2^*}\left(\frac{2^*-q}{2^*-2}\right)^{\frac{2^*-q}{q-2}}=1,
$$
which together with $\|v_\lambda\|_2^2=\frac{2(2^*-q)}{q(2^*-2)}\|v_\lambda\|_q^q$  yield the last relation. 
\end{proof}

Recall that $m_\lambda=m_\lambda^*$ for $\lambda>0$ by Lemma \ref{l22}.  Moreover, the following result holds.

\begin{lemma}\label{l211}
$m_0=m_0^*$.
\end{lemma}

\begin{proof}
Clearly,  we have 
 $$
 m_0=\inf_{u\in D^{1}(\mathbb R^N)\setminus \{0\}}\sup_{t>0}J_0(tu)\le \inf_{u\in H^1(\mathbb R^N)\setminus \{0\}}\sup_{t>0}I_0(tu)=m_0^*.
 $$
To prove the opposite inequality, we argue as in the proof of Lemma \ref{l26} and Lemma \ref{l28}, but easier. 
\end{proof}

 Clearly, Lemma \ref{l211} implies that $m^*_0$ is not attained on $\mathcal M_0$. In fact,  it is also well known that $(P_0)$ has no nontrivial solution by the Pohozaev's identity. 
Observe that
 $$
I_0(u_\lambda)=  
I_\lambda(u_\lambda)+\frac{\lambda}{q}\int_{\mathbb R^N}|u_\lambda|^q=m_\lambda+o(1)=m_0^*+o(1),
$$
and 
$$
I'_0(u_\lambda)v=I'_\lambda(u_\lambda)v+\lambda\int_{\mathbb R^N}|u_\lambda|^{q-2}u_\lambda v=o(1).
$$
That is, the family $\{u_\lambda\}$ of ground states of $(P_\lambda)$ is a $(PS)$ sequence of $I_0$ at level $m_0^*$
(otherwise $u_0$ should be a nontrivial solution of $(P_0)$, which is a contradiction).

\section{Proof of Theorem \ref{t1}}\label{s4}

We recall the P.-L.~Lions' concentration--compactness lemma, which is at the core of our proof of  Theorem \ref{t1}.

\begin{lemma}[P.-L.~Lions \cite{Lions-1}]\label{l31}
Let $r>0$ and $2\leq q\leq 2^{*}$. If $(u_{n})$ is bounded in $H^{1}(\mathbb{R}^N)$ and if
$$\sup_{y\in\mathbb{R}^N}\int_{B_{r}(y)}|u_{n}|^{q}\to0\quad\textrm{as}\ n\to\infty,$$
then $u_{n}\to0$ in $L^{p}(\mathbb{R}^N)$ for $2<p<2^*$. Moreover, if $q=2^*$, then $u_{n}\to0$ in $L^{2^{*}}(\mathbb{R}^N)$.
\end{lemma}

Using Lemma \ref{l31}, we establish the following.
 
\begin{lemma}\label{l34}
If $N\ge 5$, then $v_\lambda\to U_{\rho_0}$ in $H^1(\mathbb R^N)$ as $\lambda\to 0$, where $U_{\rho_0}$ is  a positive ground state of $(Q_0)$ with
$$
\rho_0=\left(\frac{2(2^*-q)\int_{\mathbb R^N}|U_1|^q}{q(2^*-2)\int_{\mathbb R^N}|U_1|^2}\right)^{\frac{2^*-2}{2(q-2)}}.
$$
If $N=4$ and $N=3$, then there exists $\xi_\lambda\in (0,+\infty)$ such that $\xi_\lambda\to 0$ and
$$
v_\lambda-\xi_\lambda^{-\frac{N-2}{2}}U_1(\xi^{-1}_\lambda\cdot)\to 0
$$
in $D^{1}(\mathbb R^N)$ and $L^{2^*}(\mathbb R^N)$ as $\lambda\to 0$.
\end{lemma}

\begin{proof}
Note that $v_\lambda$ is a positive radially symmetric function, and by Lemma \ref{l24}, $\{v_\lambda\}$ is bounded in $H^1(\mathbb R^N)$. Then there exists $v_0\in H^1(\mathbb R^N)$ verifying $-\Delta v=v^{2^*-1}$ such that 
\begin{equation}\label{e31}
v_\lambda\rightharpoonup v_0   \quad {\rm weakly \ in} \  H^1(\mathbb R^N), \quad v_\lambda\to v_0 \quad {\rm in} \ L^p(\mathbb R^N) \quad {\rm for \ any} \ p\in (2,2^*),
\end{equation}
and 
\begin{equation}\label{e32}
v_\lambda(x)\to v_0(x) \quad \text{a.e.~on $\R^N$},  \qquad v_\lambda\to v_0 \quad\text{in}  \  L^2_{loc}(\mathbb R^N).
\end{equation}
Observe that
$$
J_0(v_\lambda)=J_\lambda(v_\lambda)+\frac{\lambda^\sigma}{q}\int_{\mathbb R^N}|v_\lambda|^q-\frac{\lambda^\sigma}{2}\int_{\mathbb R^N}|v_\lambda|^2=m_\lambda+o(1)=m_0+o(1),
$$
and 
$$
J'_0(v_\lambda)v=J'_\lambda(v_\lambda)v+\lambda^\sigma\int_{\mathbb R^N}|v_\lambda|^{q-2}v_\lambda v-\lambda^\sigma\int_{\mathbb R^N}v_\lambda v=o(1).
$$
Therefore, $\{v_\lambda\}$ is a $(PS)$ sequence for $J_0$.

By Lemma \ref{l31},  it is standard to show that there exists $\zeta^{(j)}_\lambda\in (0,+\infty)$, $v^{(j)}\in D^{1,2}(\mathbb R^N)$ with $j=1,2,\dots, k$ where $k$  is a non-negative integer, such that
\begin{equation}\label{e33}
v_\lambda=v_0+\sum_{j=1}^k(\zeta^{(j)}_\lambda)^{-\frac{N-2}{2}}v^{(j)}((\zeta^{(j)}_\lambda)^{-1} x)+\tilde v_\lambda,
\end{equation}
where $\tilde v_\lambda\to 0$ in $L^{2^*}(\mathbb R^N)$, $v^{(j)}$ are nontrivial solutions of the limit equation $-\Delta v=v^{2^*-1}$ and  $\int_{\mathbb R^N}|\nabla v^{(j)}|^2\ge S^{\frac{N}{2}}$ with $S$ being the best  Sobolev constant. Moreover, we have 
\begin{equation}\label{e34}
\liminf_{\lambda\to 0}\|v_\lambda\|^2_{D^1(\mathbb R^N)}\ge  \|v_0\|^2_{D^1(\mathbb R^N)}+\sum_{j=1}^k\|v^{(j)}\|^2_{D^1(\mathbb R^N)},
\end{equation}
and 
\begin{equation}\label{e35}
m_0=J_0(v_0)+\sum_{j=1}^kJ_0(v^{(j)}).
\end{equation}
Moreover, $J_0(v_0)\ge 0$ and $J_0(v^{(j)})\ge m_0$ for all $j=1,2,\cdots, k.$

If $N\ge 5$, then by Lemma \ref{l210}, we have $v_0\not=0$ and hence $J_0(v_0)=m_0$ and  $k=0$.  Thus $v_\lambda\to v_0$ in $L^{2^*}(\mathbb R^N)$. Since $J_0'(v_\lambda)\to 0$, it follows that
$v_\lambda\to v_0$ in $D^{1}(\mathbb R^N)$.

Observe that by the Strauss' $H^1$--radial lemma \cite[Lemma A.II]{Berestycki-1} we have 
$$
v_\lambda(x)\le C_N|x|^{-\frac{N-1}{2}}\|v_\lambda\|_{H^1(\mathbb R^N)} \quad {\rm for} \  |x|>0.
$$
Hence we obtain
$$
\Big(-\Delta-C|x|^{-\frac{2(N-1)}{N-2}}\Big)v_\lambda\le \big(-\Delta +\lambda^\sigma-v_\lambda^{2^*-2}-\lambda^\sigma v_\lambda^{q-2}\big)v_\lambda=0,
$$
for some constant $C>0$ which is independent of $\lambda$.  We also have
$$
\Big(-\Delta-C|x|^{-\frac{2(N-1)}{N-2}}\Big)\frac{1}{|x|^{N-2-\varepsilon_0}}=\big(\varepsilon_0(N-2-\varepsilon_0)-C|x|^{-\frac{2}{N-2}}\big)\frac{1}{|x|^{N-\varepsilon_0}},
$$
which is positive for $|x|$ large enough. By the maximum principle on $\mathbb R^N\setminus B_R$, we deduce that 
\begin{equation}\label{e36}
v_\lambda(x)\le \frac{v_\lambda(R)R^{N-2-\varepsilon_0}}{|x|^{N-2-\varepsilon_0}}  \quad\text{for}  \ |x|\ge R.
\end{equation}
When $\varepsilon_0>0$ is small enough, the right hand side is in $L^2(B_R^c)$ for $N\ge 5$ and by the dominated convergence theorem we conclude that $v_\lambda\to v_0$ in $L^2(\mathbb R^N)$, and hence in  $H^1(\mathbb R^N)$. Moreover, by \eqref{e28} we obtain
$$
\int_{\mathbb R^N}|v_0|^2=\frac{2(2^*-q)}{q(2^*-2)}\int_{\mathbb R^N}|v_0|^q,
$$
from which it follows that $v_0=U_{\rho_0}$ with $$
\rho_0=\left(\frac{2(2^*-q)\int_{\mathbb R^N}|U_1|^q}{q(2^*-2)\int_{\mathbb R^N}|U_1|^2}\right)^{\frac{2^*-2}{2(q-2)}}.
$$

If $N=4$ or $3$, then by Fatou's lemma we have $\|v_0\|^2_2\le \liminf_{\lambda\to 0}\|v_\lambda\|_2^2<\infty$. Therefore,  $v_0=0$ and hence $k=1$. Thus, we obtain  
$J_0(v^{(1)})=m_0$ and hence $v^{(1)}=U_\rho$ for some $\rho\in (0,+\infty)$. Therefore, we conclude that 
$$v_\lambda-\xi_\lambda^{-\frac{N-2}{2}}U_1(\xi_\lambda^{-1}\cdot )\to 0
$$ in $L^{2^*}(\mathbb R^N)$ as $\lambda\to 0$, where 
$\xi_\lambda:=\rho\zeta_\lambda^{(1)}\in (0,+\infty)$ satisfying $\xi_\lambda\to 0$ as $\lambda\to 0$.
Since $$J_0'(v_\lambda-\xi_\lambda^{-\frac{N-2}{2}}U_1(\xi_\lambda^{-1}\cdot ))=J'_0(v_\lambda)+J'_0(U_1)+o(1)=o(1)$$ as $\lambda\to 0$, it follows that $v_\lambda-\xi_\lambda^{-\frac{N-2}{2}}U_1(\xi_\lambda^{-1}\cdot )\to 0$ in $D^{1}(\mathbb R^N)$
\end{proof}

In the lower dimension cases $N=4$ and $N=3$, we perform an additional rescaling
\begin{equation}\label{e37}
w(x)=\xi_\lambda^{\frac{N-2}{2}} v(\xi_\lambda x), 
\end{equation}
where $\xi_\lambda\in (0,+\infty)$ is given in Lemma \ref{l34}. This rescaling transforms $(Q_\lambda)$ into an equivalent equation
$$
-\Delta w+\lambda^\sigma\xi_\lambda^{(2^*-2)s}w=w^{2^*-1}+\lambda^\sigma\xi_\lambda^{(2^*-q)s}w^{q-1}
\quad\text{in} \ \R^N,
\leqno(R_\lambda)
$$
here and in what follows, we set for brevity
$$
s:=\tfrac{N-2}{2}=\left\{\begin{array}{rcl} 1, \quad\text{if}  \ N=4,\smallskip\\
\frac{1}{2}, \quad\text{if} \  N=3.
\end{array}\right.
$$
The  corresponding energy functional is given by
\begin{equation}\label{e38}
\tilde J_\lambda(w):=\frac{1}{2}\int_{\mathbb R^N}|\nabla w|^2+\lambda^\sigma\xi_\lambda^{(2^*-2)s}|w|^2-\frac{1}{2^*}\int_{\mathbb R^N}|w|^{2^*}-\frac{1}{q}\lambda^\sigma\xi_\lambda^{(2^*-q)s}\int_{\mathbb R^N}|w|^q.
\end{equation}

It is straightforward to verify the following.

\begin{lemma}\label{l36}
Let $\lambda>0$, $u\in H^1(\mathbb R^N)$ and $v$ and $w$ are the rescalings  \eqref{e22} and \eqref{e37} of $u$ respectively. Then:

\begin{enumerate}
	\item[$(a)$] $\|\nabla w\|_2^2= \|\nabla v\|_{2}^{2}=\|\nabla u\|_{2}^{2}$, $\|w\|^{2^*}_{2^*}=\|v\|_{2^*}^{2^*}=\|u\|_{2^*}^{2^*}$, \smallskip

\item[$(b)$]  $\xi_\lambda^{(2^*-2)s}\|w\|^2_2=\|v\|_2^2=\lambda^{-\sigma}\| u\|_2^2$,    $\xi_\lambda^{(2^*-q)s}\|w\|^q_q=\|v\|_q^q=\lambda^{1-\sigma} \|u\|_q^q$, \smallskip

\item[$(c)$]  $\tilde J_\lambda(w)=J_\lambda(v)=I_\lambda(u)$. 
\end{enumerate}
\end{lemma}

Let $w_\lambda(x)=\xi_\lambda^{\frac{N-2}{2}} v_\lambda(\xi_\lambda x)$ where the $v_\lambda$ is a ground state of $(Q_\lambda)$. Then by Lemma \ref{l34} we conclude that 
\begin{equation}\label{e39}
\|\nabla(w_\lambda-U_1)\|_2\to 0, \qquad \|w_\lambda-U_1\|_{2^*}\to 0  \qquad\text{as} \ \lambda\to 0.
\end{equation}

Note that the corresponding Nehari and Pohozaev's identities read as follows
$$
\int_{\mathbb R^N}|\nabla w_\lambda|^2+\lambda^\sigma\xi_\lambda^{(2^*-2)s}\int_{\mathbb R^N}|w_\lambda|^2=\int_{\mathbb R^N}|w_\lambda|^{2^*}+\lambda^\sigma\xi_\lambda^{(2^*-q)s}\int_{\mathbb R^N}|w_\lambda|^q,
$$
and 
$$
\frac{1}{2^*}\int_{\mathbb R^N}|\nabla w_\lambda|^2+\frac{1}{2}
\lambda^\sigma\xi_\lambda^{(2^*-2)s}\int_{\mathbb R^N}|w_\lambda|^2=\frac{1}{2^*}\int_{\mathbb R^N}|w_\lambda|^{2^*}+\frac{1}{q}\lambda^\sigma\xi_\lambda^{(2^*-q)s}\int_{\mathbb R^N}|w_\lambda|^q.
$$
We conclude that 
$$
\left(\frac{1}{2}-\frac{1}{2^*}\right)\lambda^\sigma\xi_\lambda^{(2^*-2)s}\int_{\mathbb R^N}|w_\lambda|^2=\left(\frac{1}{q}-\frac{1}{2^*}\right)\lambda^\sigma\xi_\lambda^{(2^*-q)s}\int_{\mathbb R^N}|w_\lambda|^q.
$$
Thus, we obtain
\begin{equation}\label{e310}
\xi_\lambda^{(q-2)s}\int_{\mathbb R^N}|w_\lambda|^2=\frac{2(2^*-q)}{q(2^*-2)}\int_{\mathbb R^N}|w_\lambda|^q.
\end{equation}
To control the norm $\|w_\lambda\|_2$, we note that  for any $\lambda>0$, $w_\lambda>0$ satisfies the linear inequality
\begin{equation}\label{e311}
-\Delta w_\lambda+\lambda^\sigma\xi_\lambda^{(2^*-2)s}w_\lambda=w_\lambda^{2^*-1}+\lambda^\sigma\xi_\lambda^{(2^*-q)s}w_\lambda^{q-1}>0,  \quad x\in \mathbb R^N.
\end{equation}

\begin{lemma}\label{l37}
There exists a constant $c>0$ such that 
\begin{equation}\label{e312}
w_\lambda(x)\ge c|x|^{-(N-2)}\exp({-\lambda^{\frac{\sigma}{2}}\xi_\lambda^{\frac{(2^*-2)s}{2}}}|x|),  \quad |x|\ge 1.
\end{equation}
\end{lemma}

\begin{proof} The same as \cite[Lemma 4.8]{Moroz-1}.  
\end{proof}

As consequences, we have the following two lemmas.

\begin{lemma}\label{l38}
If $N=3$, then $\|w_\lambda\|_2^2\gtrsim \lambda^{-\frac{\sigma}{2}}\xi_\lambda^{-\frac{(2^*-2)s}{2}}$.
\end{lemma}

\begin{lemma}\label{l39}
If $N=4$, then $\|w_\lambda\|_2^2\gtrsim  - \ln(\lambda^{\sigma}\xi_\lambda^{(2^*-2)s})$.
\end{lemma}

To prove our main result, the key point is to show the boundedness of $\|w_\lambda\|_q$. 

\begin{lemma}\label{l311}
If $N=3,4$ and  $\frac{N}{N-2}<r<2^*$, then $\|w_\lambda\|_r^r\sim 1$ as $\lambda\to 0$. Furthermore, $w_\lambda\to U_1$ in $L^r(\mathbb R^N)$ as $\lambda\to 0$. 
\end{lemma}

\begin{proof}
By \eqref{e39},  we have $w_\lambda\to U_1$ in $L^{2^*}(\mathbb R^N)$. Then, as in \cite[Lemma 4.6]{Moroz-1}, using the embeddings $L^{2^*}(B_1)\hookrightarrow L^r(B_1)$ we prove that
$\liminf_{\lambda\to 0}\|w_\lambda\|_r^r>0$.  
 
 On the other hand,  
arguing as in \cite[Propositon 3.1]{Akahori-2}, we show that there exists a constant $C>0$ such that for all small $\lambda>0$, 
\begin{equation}\label{e314}
w_\lambda(x)\le \frac{C}{(1+|x|)^{N-2}}, \qquad \forall x\in \mathbb R^N,
\end{equation}
which together with the fact that $r>\frac{N}{N-2}$ implies that $w_\lambda$ is bounded in $L^r(\mathbb R^N)$  uniformly for small $\lambda>0$,  and 
 by the dominated convergence theorem $w_\lambda\to U_1$ in $L^r(\mathbb R^N)$ as $\lambda\to 0$. 
\end{proof}

\begin{proof}[Proof of Theorem \ref{t1}]   We only give the proof for $N=3, 4$. The case $N\ge 5$ is easier. 
We first note that for a result similar to Lemma \ref{l24} holds for $w_\lambda$ and $\tilde J_\lambda$. 
By \eqref{e310}, \eqref{e25} and Lemma \ref{l36}, we also have  $\tau(w_\lambda)=\tau(v_\lambda)$. 
Therefore, by \eqref{e310} we obtain
\begin{align}\label{e318}
m_0&\le\sup_{t\ge 0} \tilde J_\lambda((w_\lambda)_t)+\lambda^\sigma\tau(w_\lambda)^{\frac{N}{2}}\left\{\frac{1}{q}\xi_\lambda^{(2^*-q)s}\int_{\mathbb R^N}|w_\lambda|^q-\frac{1}{2}\xi_\lambda^{(2^*-2)s}\int_{\mathbb R^N}|w_\lambda|^2\right\}\\
&=m_\lambda+\lambda^\sigma\tau(v_\lambda)^{\frac{N}{2}}\frac{q-2}{q(2^*-2)}\xi_\lambda^{(2^*-q)s}\int_{\mathbb R^N}|w_\lambda|^q,\nonumber
\end{align}
which implies that
$$
\xi_\lambda^{(2^*-q)s}\int_{\mathbb R^N}|w_\lambda|^q\ge \lambda^{-\sigma}\frac{q(2^*-2)}{(q-2)\tau(v_\lambda)^{\frac{N}{2}}}\delta_\lambda,
$$
where $\delta_\lambda=m_0-m_\lambda$. Hence, by Corollary \ref{c29}, we obtain
\begin{equation}\label{e319}
\xi_\lambda^{(2^*-q)s}\int_{\mathbb R^N}|w_\lambda|^q\gtrsim \lambda^{-\sigma}\delta_\lambda\gtrsim 
\left\{\begin{array}{ll} 
 (\ln\frac{1}{\lambda})^{-\frac{4-q}{q-2}}  &\text{if}\ N=4,\smallskip\\
 \lambda^{\frac{2(6-q)}{(q-2)(q-4)}}   &\text{if}\  N=3.
 \end{array}\right.
\end{equation}  
Therefore, by Lemma \ref{l311}, we have 
\begin{equation}\label{e320}
\xi_\lambda\gtrsim  \left\{\begin{array}{ll} 
 (\ln\frac{1}{\lambda})^{-\frac{1}{q-2}}&\text{if}\ N=4,\\
 \lambda^{\frac{4}{(q-2)(q-4)}}&\text{if}\ N=3.
 \end{array}\right.
\end{equation}
On the other hand, if  $ N=3$, then by  \eqref{e310}, Lemma \ref{l38} and Lemma \ref{l311}, we have 
$$
\xi_\lambda^{(q-2)s}\lesssim \frac{1}{\|w_\lambda\|_2^2}\lesssim\lambda^{\frac{\sigma}{2}}\xi_\lambda^{\frac{(2^*-2)s}{2}}.
$$
Then 
$$
\xi_\lambda^{(q-4)s}\lesssim \lambda^{\frac{\sigma}{2}}.
$$
Hence, observing that $s=\frac{N-2}{2}=\frac{1}{2}$, $\sigma=\frac{2^*-2}{q-2}=\frac{4}{q-2}$,  for $q\in (4,6)$ we obtain 
\begin{equation}\label{e321}
\xi_\lambda\lesssim \lambda^{\frac{4}{(q-2)(q-4)}}.
\end{equation}
If $N=4$, then by \eqref{e310}, Lemma \ref{l39} and Lemma \ref{l311},  we have 
$$
\xi_\lambda^{(q-2)s}\lesssim \frac{1}{\|w_\lambda\|_2^2}\lesssim \frac{1}{-\ln(\lambda^\sigma\xi_\lambda^{(2^*-2)s})}.
$$
Note that 
$$
-\ln(\lambda^\sigma\xi_\lambda^{(2^*-2)s})=\sigma\ln\frac{1}{\lambda}+(2^*-2)s\ln\frac{1}{\xi_\lambda}\ge \sigma\ln\frac{1}{\lambda},
$$
it follows that 
$$
\xi_\lambda^{(q-2)s}\lesssim  \frac{1}{\|w_\lambda\|_2^2}\lesssim \Big(\ln\frac{1}{\lambda}\Big)^{-1}.
$$
Since $s=\frac{N-2}{2}=1$,  we then obtain 
\begin{equation}\label{e322}
\xi_\lambda\lesssim   \Big(\ln\frac{1}{\lambda}\Big)^{-\frac{1}{q-2}}.
\end{equation}
Thus, it follows from \eqref{e318}, \eqref{e321}, \eqref{e322} and Lemma \ref{l311} that 
$$
\delta_\lambda=m_0-m_\lambda\lesssim \lambda^\sigma\xi_\lambda^{(2^*-q)s}\lesssim 
\left\{\begin{array}{ll} \lambda^{\frac{2}{q-2}}(\ln\frac{1}{\lambda})^{-\frac{4-q}{q-2}}&\text{if}\ N=4,\\
 \lambda^{\frac{2}{q-4}}&\text{if}\ N=3,
 \end{array}\right.
$$
which together with Corollary \ref{c29} implies that 
$$
\|\nabla U_1\|_2^2-\|\nabla w_\lambda\|_2^2=N\delta_\lambda\sim 
\left\{\begin{array}{ll} \lambda^{\frac{2}{q-2}}(\ln\frac{1}{\lambda})^{-\frac{4-q}{q-2}} &\text{if}\ N=4,\\
 \lambda^{\frac{2}{q-4}}&\text{if}\  N=3.
 \end{array}\right.
$$
Finally, by \eqref{e310}, Lemma \ref{l38} and Lemma \ref{l39}, we obtain
$$
\|w_\lambda\|_2^2\sim 
\left\{\begin{array}{ll}
\ln\frac{1}{\lambda}&\text{if}\ N=4,\\
\lambda^{-\frac{2}{q-4}}&\text{if}\  N=3.
\end{array}\right.
$$
Statements on $u_\lambda$ follow from the corresponding results on $v_\lambda$ and $w_\lambda$. This completes the proof of Theorem \ref{t1}.
\end{proof}

\bigskip
{\small
\noindent {\bf Acknowledgements.} 
The authors are grateful to the anonymous referee for their multiple helpful suggestions.

Part of this research was carried out while S.M. was visiting Swansea University. S.M. thanks the Department of Mathematics for its hospitality.
S.M. was supported by National Natural Science Foundation of China
(Grant Nos.11571187, 11771182)

\end{document}